\documentclass[12pt,a4paper,reqno]{amsart}
\usepackage[utf8]{inputenc}
\usepackage[T1]{fontenc}
\usepackage[english]{babel}
\usepackage{amsfonts,amsmath, amssymb,latexsym,times,mathrsfs}

\usepackage{float}
\usepackage{comment}
\usepackage{ae,aecompl}
\usepackage{tikz-cd}

\usepackage[colorlinks=true]{hyperref}

\numberwithin{equation}{section}
\def\TT{{\mathbb{T}}}

\def\NN{{\mathbb{N}}}
 
\def\ZZ{{\mathbb{Z}}}  \def\SS{{\mathbb{S}}}   \def\CC{{\mathbb{C}}}
\def\RR{{\mathbb{R}}}  \def\TT{{\mathbb{T}}}

\def\cG{{\mathcal{G}}}

\def\cJ{{\mathcal{J}}}

\def\scrO{{\mathscr{O}}}
\def\scrQ{{\mathcal{Q}}}

\def\scrT{{\mathscr{T}}}

\def\scrF{{\mathscr{F}}}
\def\scrK{{\mathscr{K}}}
\def\scrR{{\mathscr{R}}}

\def\scrM{{\mathscr{M}}}

\def\scrU{{\mathscr{U}}}
\def\scrV{{\mathscr{V}}}

\def\vecV{\vec V}
\renewcommand{\epsilon}{\varepsilon}
\newcommand{\ip}[1]{\langle #1 \rangle}
\newcommand{\ipp}[1]{\langle \!  \langle #1 \rangle\! \rangle}
\newcommand{\LieTT}{\mathfrak{t}}

\newcommand{\area}{\mathrm{area}}

\newcommand{\id}{\mathrm{id}}
\newcommand{\Diff}{\mathrm{Diff}}
\newcommand{\Ham}{\mathrm{Ham}}

\newcommand{\End}{\mathrm{End}}

\newcommand{\del}{\partial}

\newtheorem{theointro}{Theorem}

\newtheorem{conj}[theointro]{Conjecture}

\newtheorem{corintro}[theointro]{Corollary}

\newtheorem{lemma}[subsubsection]{Lemma}
\newtheorem{prop}[subsubsection]{Proposition}
\newtheorem{cor}[subsubsection]{Corollary}
\newtheorem{theo}[subsubsection]{Theorem}

\newtheorem{dfn}[subsubsection]{Definition}

\theoremstyle{definition}

\newtheorem{rmk}[subsubsection]{Remark}

\newtheorem{question}[subsubsection]{Question}
\newtheorem*{rmk*}{Remark}
\newtheorem*{rmks*}{Remarks}
\newtheorem{rmks}[subsubsection]{Remarks}

\thanks{Supported by  IRL CRM-CNRS 3457, UQAM, CIRGET, UMR6629,
ANR-21-CE40-0017, CHL 11-LABX-0020-01}

\title{Isotropic surfaces and moment map flow}

\author{François Jauberteau}
\address{François Jauberteau, Laboratoire Jean Leray, Universit\'e de Nantes}
\email{francois.jauberteau@univ-nantes.fr}

\author{Yann Rollin}
\address{Yann Rollin, Laboratoire Jean Leray, Universit\'e de Nantes}
\email{yann.rollin@univ-nantes.fr}

\begin{document}

\begin{abstract}
We consider the moduli  space of isotropic  maps from a  closed
	surface $\Sigma$ to a symplectic affine space and  construct a Kähler moment
	map geometry, on a space of differential forms on $\Sigma$, such that
	the isotropic maps correspond to certain zeroes of the moment map.
	The moment map geometry induces a modified moment map flow, whose
	fixed point set
	correspond to isotropic maps. 
	This construction can be adapted to
	the polyhedral setting. In particular, we prove that the polyhedral
	modified moment map flow induces a strong deformation retraction from the space of polyhedral
	maps onto the space of polyhedral isotropic maps.
\end{abstract}

{\Huge \sc \bf\maketitle}

\section{Introduction}

\subsection{Motivations}
\emph{Symplectic geometry} is the natural mathematical framework for
  \emph{Hamiltonian mechanics}.
Gromov showed that certain symplectic properties
are \emph{flexible},
thanks to \emph{convex integration},
 resulting in \emph{h-principle} theorems~\cite{Gro85}, whereas others are \emph{rigid}, due to
the existence of \emph{pseudoholomorphic curves}~\cite{Gro86}.
In particular the  Gromov-Lees theorem~\cite{Gro86,Lee76,EM02}  establishes an h-principle for the
problem of \emph{isotropic immersions} in a symplectic manifold. 
As a corollary, it is possible to approximate every immersed submanifold of an
affine symplectic space by isotropic immersed submanifolds, with respect to
the
$C^0$ topology.

Our research started as we were
conducting some numerical experiments for
Lagrangian surfaces and their \emph{mean curvature flow}; numerical
simulations naturally involve piecewise linear geometry and
we quickly realized that 
the only  explicit examples of closed piecewise linear Lagrangian 
were the polygons in $\CC$ and some polygonal versions of the Clifford
torus in $\CC^m$.
Assessing the state of the art, we noticed
that very little is known about piecewise
linear symplectic geometry. 
 Then, we
turned to the  existence problem for piecewise
linear isotropic submanifolds in  affine symplectic spaces~\cite{JRT,Rol22}
and  obtained the following approximation theorem:
\begin{theo}[\cite{Rol22}]
	\label{theo:approx}
	Let $f:\Sigma\looparrowright V$ be a smooth isotropic immersion, where
	$\Sigma$ is  a surface diffeomorphic to a quotient torus, and
	$(V,\omega_V)$ is a symplectic affine
	space. Then, there exists a family of piecewise linear maps
	$f_N:\Sigma\to V$ for $N\in\NN$, with the following
	properties: 
	\begin{enumerate}
		\item $f_N:\Sigma\looparrowright V$ is a  topological
	immersion;
	\item $f_N$ converges toward $f$ in $C^1$-norm.
	\end{enumerate}

As a corollary, every smoothly immersed torus in $V$ can be approximated in
$C^0$-norm by isotropic piecewise linear immersed tori.
\end{theo}
 Surprisingly, the original proof of Theorem~\ref{theo:approx} relies on an
 \emph{infinite dimensional Kähler moment map 
 geometry} and the \emph{fixed point principle}
 developed in \cite{JRT}.
However, an alternate proof using only on \emph{soft techniques}
 was given subsequently by 
Etourneau~\cite{Eto}, with significant technical simplifications.

Several
authors have been  considering problems of piecewise linear symplectic
geometry,
including Gratza~\cite{Gra} and Bertelson-Distexhe~\cite{BD}.
Panov also introduced 
a notion of \emph{polyhedral Kähler} geometry in~\cite{Pan}. 
The space of piecewise linear symplectic maps of a $4$-torus is also
shown to be related to a \emph{hyperKähler moment map
geometry} in \cite{Rol24}.
It seems that classical results of smooth symplectic differential
geometry are curiously challenging to generalize to the piecewise linear setting.
Perhaps this is a manifestation of symplectic rigidity ? 
To show how much has to be accomplished in the field, here is a short list of challenging open questions:
\begin{enumerate}
	\item The \emph{local Darboux theorem}, well known in the smooth setting,
turns out to be a conjecture for piecewise linear symplectic
		manifolds of dimension at least $4$ (cf.~\cite{BD} for a definition of piecewise
		linear symplectic manifolds).
	\item The \emph{deformation theory of Lagrangian submanifolds} is well described
	in the smooth setting, thanks to the Lagrangian neighborhood
		theorem~\cite{MDS}.
		However, there is no analogue of such a deformation theory
		in the case of
		piecewise linear Lagrangian submanifolds of a symplectic
		affine space. This  motivates the work and partial answers
		of \cite{JRT,Rol22}.
\item
	The \emph{local structure} of the infinite dimensional Lie group of symplectic diffeomorphisms
		is well
		understood~\cite{Ban}. However, very little is known about
		the local (and global) structure of the space of piecewise linear
symplectic maps, even in the case of a $4$-dimensional torus~\cite{Rol24}.
\end{enumerate}
The above questions seem hard settle for an obvious reason: the Moser's
trick generally fails  in the piecewise linear setting, although it may be
be used  with significant efforts, in the case of piecewise linear volume
forms for instance~\cite{BD} for instance. All these complications lead us
to speculate whether  
some \emph{exotica} arise in the context of piecewise linear symplectic
geometry. 

A Kähler moment map geometry inspired by Donaldson~\cite{Don99} was introduced
in \cite{JRT} to construct the approximations $f_N$ of Theorem~\ref{theo:approx}.
The construction starts from a  Kähler surface
$(\Sigma,g_\Sigma,J_\Sigma,\omega_\Sigma)$ and a Hermitian affine space $V$
with symplectic form $\omega_V$.
These structures induce a 
\emph{formal} Kähler  structure $(\scrM,\cG,\cJ,\Omega)$ with Kähler form
$\Omega$, on the moduli
space~$\scrM=C^\infty(\Sigma,V)$.
The group of Hamiltonian diffeomorphisms $\Ham(\Sigma,\omega_\Sigma)$
acts by precomposition on the moduli space space $\scrM$ and preserves the formal Kähler structure.
In fact the action of $\Ham(\Sigma,\omega_\Sigma)$ on $\scrM$ is formally Hamiltonian,
with moment map given by 
$$\mu^D(f)=-
\frac{f^*\omega_V}{\omega_\Sigma},
$$
understood as an element of the Lie algebra of
$\Ham(\Sigma,\omega_\Sigma)$, identified to the space of smooth function
$C_0^\infty(\Sigma,\RR)$ orthogonal to constants.
Thus, zeroes of the moment map $\mu^D$ correspond to 
isotropic maps $f\in\scrM$. A detailed presentation of this moment map
geometry is given 
in~\cite{JRT}.

A finite dimensional approximation $\phi_N$ of the energy functional $\phi$ of the moment
map $\mu^D$ is considered in \cite{JRT}. The downward gradient flow of
$\phi_N$ provides a flow of piecewise linear surfaces.
A numerical version of the flow  and numerical
experiments  carried out in \cite{JRT} provided some  effective examples of piecewise
linear Lagrangian tori in $\CC^2$. Overall, this flow 
of piecewise linear surfaces  has been interesting from an
experimental perspective. However, it is not satisfactory from a
mathematical point of view, for several reasons:
we could not prove that the flow is convergent, although it seems well
behaved numerically. Furthermore, the flow does not come from
 a finite dimensional moment map picture and its
geometrical interpretation is somewhat unclear. 

As a response to these objections,
a new moment map geometry and its
corresponding flow are introduced in this paper. They do
not have any of the above issues: the new moment map geometry can be
immediately adapted to the polyhedral setting and much stronger mathematical results
are obtained. In
particular, Theorem~\ref{theo:main} is a Duistermaat type theorem,
which shows that the  flow has the nicest
possible behavior.

\subsection{Statement of results}
Let $V$ be a Hermitian affine space, with underlying vector space
$\vecV$ and 
and \emph{symplectic form} $\omega_V$.
We consider a \emph{smooth closed  oriented surface} $\Sigma$, endowed with
a \emph{Riemannian
metric} $g_\Sigma$, a \emph{compatible almost complex structure} $J_\Sigma$
and a
corresponding \emph{Kähler form} $\omega_\Sigma=g_\Sigma(J_\Sigma,\cdot,\cdot)$.

The \emph{moduli space of smooth map} $f:\Sigma\to V$ is denoted $\scrM$.
Formally $\scrM$ is an affine space with $\Omega^0(\Sigma,\vecV)$ as the
underlying vector space.
The \emph{differential} of a smooth map $f\in\scrM$ defines an exact
$\vecV$-valued $1$-form $df$ and we have a linear differential operator
$$
d:\scrM\to\scrF=\Omega^1(\Sigma,\vecV).
$$
The image of $\scrM$ by $d$ is the space of exact $1$-forms denoted
$\scrF_0$.
We show that the moduli space $\scrF$ carries a natural Euclidean $L^2$-inner product
$\cG$, defined by Formula~\eqref{eq:dfnGmetric}, 
and a compatible almost complex structure
$\cJ$, defined by Formula~\eqref{eq:dfnJ}. The corresponding Kähler form is
denoted $\Omega=\cG(\cJ\cdot,\cdot)$.

Formula~\eqref{eq:cxaction} defines an action of the gauge group
$\TT^\CC=C^\infty(\Sigma,\CC^*)$ on $\scrF$. The action of the
real subgroup $\TT= C^\infty(\Sigma,\RR)$
is merely the action by complex multiplication on $\vecV$-valued differential
$1$-forms and we have the following result:
\begin{theointro}
	\label{theo:A}
	Let $\Sigma$ be a smooth closed surface endowed with a Kähler
	structure $(\Sigma,g_\Sigma,J_\Sigma,\omega_\Sigma)$ and
	$V$, a
	Hermitian affine space.

	The moduli space $\scrF=\Omega^1(\Sigma,\vecV)$ carries a natural
	Kähler structure $(\scrF,\cG,\cJ,\Omega)$ and an action of the
	gauge group $\TT^\CC=C^\infty(\Sigma,\CC^*)$.

The almost complex structure $\cJ$ is invariant under the action
of $\TT^\CC$. The action of $\TT^\CC$ is the complexification
	of the $\TT$-action, where $\TT=C^\infty(\Sigma,S^1)$ is the real
	subgroup.

	The  $\TT$-action preserves the Kähler structure
	$(\scrF,\cJ,\cG,\Omega)$ and  is Hamiltonian.
	The map $\mu:\scrF\to\LieTT$, where $\LieTT$ is the Lie algebra of
	$\TT$, given by
	$$
	\mu(F)=-\frac {F^*\omega_V}{\omega_\Sigma},
	$$
	is a moment map. In other words, $\mu$ is $\TT$-invariant and, for every $\zeta\in\LieTT$,
$$
	D\ipp{\mu,\zeta}=-\iota_{X_\zeta}\Omega,
$$
where $X_\zeta$ is the vector field on $\scrF$ defined by the infinitesimal action of $\zeta$
	on $\scrF$.
\end{theointro}
\begin{corintro}
	\label{cor:A}
	With the assumptions of Theorem~\ref{theo:A},
a smooth map $f:\Sigma\to V$ is isotropic if, and only if,
$$
	\mu(df)=0.
$$
\end{corintro}
\begin{rmk}
In the finite dimensional setting, the famous Kempf-Ness theorem relates
	the existence of zeroes of a Kähler moment
map in a complexified orbit with an algebro-geometric notion of
	stability~\cite{KFM}.
	In view Theorem~\ref{theo:A} and Corollary~\ref{cor:A}, it is
	tempting to try to extend this theory in our case:
	the existence of isotropic maps should be related to some kind of algebraic
	condition. The answer turns out to be somewhat trivial for the
	total moduli space $\scrF$, as discussed at \S\ref{sec:kn}. However, 
	 we are mainly  interested in the zeroes of the moment map that
	 belong
	 to the subspace $\scrF_0\subset\scrF$, as they are related to
	 isotropic maps by Corollary~\ref{cor:A}. Unfortunately, $\scrF_0$  is not invariant under the
	gauge group action and it is not clear how to adapt the classical
	theory from this point.
\end{rmk}

We consider the \emph{energy functional} of the moment map
$$
\phi:\scrF\to \RR
$$
given by 
$$
\phi(F)=\frac 12 \|\mu(F)\|^2_{L^2}.
$$
By construction, $\phi$ is non negative and the vanishing locus of $\phi$ agrees with the vanishing
locus of the moment map. Our idea is to interpret the functional $\phi$
as a \emph{Morse-Bott} function on the moduli
space $\scrF$, in the spirit of Atiyah-Bott~\cite{AB}, who considered the case of the Yang-Mills functional. 
The expectation is that the flow is going to produce a \emph{strong
deformation retraction} of $\scrM$ onto the
space of isotropic maps.
However, it is not clear whether $\phi$ satisfies the Morse-Bott condition
in any reasonable sense.
Another difficulty is that the subspace $\scrF_0$ is not gauge invariant
and may not be preserved by the usual Morse-Bott gradient flow. 
We get around this issue by  defining the \emph{modified moment map flow}, as the downward gradient flow of the
restricted functional $\phi:\scrF_0\to\RR$. More precisely, the modified
moment map flow along $\scrF_0$ is given by the evolution equation
$$
\boxed{
	\frac{\del F}{\del t}=-\nabla^\circ \phi(F),
	}
$$
where $F\in\scrF_0$ and $\nabla^\circ\phi$ is the gradient of the
restricted functional.
Our first theorem shows that  the flow is geometrically relevant and is a
well posed problem from an analytical point of view:
\begin{theointro}
	\label{theo:smoothflow}
The fixed point locus of the modified moment map flow on $\scrF_0$, in other
words the critical set of $\phi:\scrF_0\to\RR$, agrees with the vanishing
	locus of $\phi:\scrF_0\to\RR$.
	Furthermore,
	\begin{enumerate}
		\item the modified moment maps flow has the short time
			existence property and
		\item the $L^2$-norm is non increasing along the flow.
	\end{enumerate}
\end{theointro}
\begin{rmks}
	\begin{enumerate}
		\item The  definition of the \emph{short time existence property} involves
	the use of
	Hölder completions of the moduli space~$\scrF$. A  technical
	version of the short time existence is given at
			Theorem~\ref{theo:lip}.
\item Unfortunately, the modified moment map flow does not seem to have any nice
	regularizing properties, like  parabolic flows for instance.
			Consequently, the long time existence of the flow
			is unclear.
		\item In the case where $\dim_\RR V=4$, we show   at Corollary~\ref{cor:mb}
			that $\phi:\scrF_0\to\RR$  is a Morse-Bott
			function, in a neighborhood  of every 
			\emph{monomorphism} (cf.
			Definition~\ref{dfn:monoiso}) of its vanishing locus.
	\end{enumerate}
\end{rmks}

The purpose of \S\ref{sec:poly} is to extend all the above 
constructions  to the \emph{polyhedral} setting.
The concepts needed for stating the results are quickly introduced below
and the reader may  refer directly to \S\ref{sec:poly} for the exact definitions.

A
\emph{polyhedral surface} $\Sigma$
is a topological surface endowed with a triangulation $\scrT$ and a \emph{polyhedral
metric} $g_\Sigma$. An \emph{orientation} of $\Sigma$ induces a  \emph{polyhedral Kähler
form} $\omega_\Sigma$ and a \emph{polyhedral almost complex structure}
$J_\Sigma$ adapted to $g_\Sigma$.

A \emph{polyhedral map} with respect to a triangulation $\scrT$ is a continuous map $f:\Sigma\to V$ such that the
restriction of $f$ to
every simplex of the triangulation is an \emph{affine map}. 
The moduli space of polyhedral maps $f:\Sigma\to V$ is  denoted 
$\scrM(\scrT)$.

A polyhedral map is generally not differentiable. However, the
restriction of a polyhedral map to any simplex of the triangulation has a
well defined tangent map.
Accordingly, a polyhedral map $f:\Sigma\to V$ is called a \emph{polyhedral isotropic map}, if the pullback of $\omega_V$ by
$f$, restricted to every simplex of the triangulation, vanishes
identically. 

The existence of  polyhedral isotropic \emph{immersions} in $\scrM(\scrT)$ is an open question.
The proof of Theorem~\ref{theo:approx}  proceeds  by introducing a
particular
sequence of triangulations $\scrT_N$, in the case of the $2$-torus $\Sigma$,
with a large number of simplices
of order $\scrO(N^2)$ and stepsize of order $\scrO(N^{-1})$. The piecewise
linear approximations $f_N$ are in fact polyhedral isotropic maps that
belong to
$\scrM(\scrT_N)$.
However, the topology of  the space of
polyhedral isotropic immersions in $\scrM(\scrT)$ remains completely
mysterious  for a fixed triangulation $\scrT$ of $\Sigma$.

Returning to the general case of a oriented polyhedral surface $\Sigma$, we
pursue the analogy with the  smooth setting. 
The space vector space $\scrF(\scrT)$ is the space of families
$F=(F_\sigma)_{\scrK_2}$, where $\sigma$ belongs to the set of facets
$\scrK_2$ of the
triangulation $\scrT$ and $F_\sigma$ is a constant $\vecV$-valued $1$-form on
$\sigma$. 
Because the restriction of a polyhedral map to every
simplex of the triangulation is differentiable, there is a natural
differential map 
$$
d:\scrM(\scrT)\to\scrF(\scrT)
$$
and its image is denoted $\scrF_0(\scrT)$.

We also define the group $\TT^\CC(\scrT)$ as the space of families
$\lambda=(\lambda_\sigma)_{\sigma\in\scrK_2}$, where
$\lambda_\sigma:\sigma\to\CC^*$ is a constant function. The real
subgroup of families $\lambda$ with $\lambda_\sigma\in S^1$ is denoted $\TT(\scrT)$.
The group $\TT^\CC(\scrT)$ acts on
$\scrF(\scrT)$ by Formula~\eqref{eq:acpol}.
As in the smooth case, we construct a Euclidean inner product $\cG$, a
compatible  almost complex structure $\cJ$ and a corresponding Kähler form
$\Omega$ on $\scrF(\scrT)$.
Thus we have a Kähler structure $(\scrF(\scrT),\cG,\cJ,\Omega)$ with an
action of $\TT^\CC(\scrT)$ and we can state our next result:
\begin{theointro}
	Theorem~\ref{theo:A} and Corollary~\ref{cor:A} hold in the polyhedral setting.
\end{theointro}
As in the smooth setting, we define an energy of the moment map
$\phi:\scrF(\scrT)\to \RR$ by $\phi(F)=\frac 12\|\mu(F)\|^2_{L^2}$.
 The downward gradient of the functional $\phi$ restricted to
$\scrF_0(\scrT)$ is called the \emph{polyhedral modified moment map flow}.

The moduli spaces $\scrM(\scrT)$ and $\scrF(\scrT)$ are finite dimensional
and much stronger result than
Theorem~\ref{theo:smoothflow} are
expected, as the polyhedral moment  map flow is an ordinary
differential equation. This is indeed the case, and we obtain the following Duistermaat type theorem:
\begin{theointro}
	\label{theo:main}
	For $F\in\scrF_0(\scrT)$,
	the polyhedral modified moment map flow 
	admits a unique solution $F_t\in\scrF_0(\scrT)$, defined for
	$t\in[0,+\infty)$, such that $F_0=F$.
	Furthermore, $F_t$ admits a limit $F_\infty\in\scrF_0(\scrT)$ as $t$ goes to
	$+\infty$, with the property that $\phi(F_\infty)=0$.

	The extended flow
	$$
	\Theta:[0,+\infty]\times\scrF_0(\scrT)\to\scrF_0(\scrT)
	$$
	defined by $\Theta(t,F)=F_t$ for $t\in[0,+\infty)$ and
	$\Theta(+\infty,F)=F_\infty$ has the following properties:
	\begin{enumerate}
		\item The fixed point locus of the flow $\Theta$ is the
			vanishing set of $\phi:\scrF_0(\scrT)\to \RR$.
		\item The flow  defines a strong deformation
	retraction of $\scrF_0(\scrT)$ onto the vanishing locus of
	$\phi:\scrF_0(\scrT)\to\RR$.
\item	There are non trivial flow lines converging toward $0$. More
	precisely, there exists $F\in\scrF_0(\scrT)\setminus 0$, such that 
		$\Theta(+\infty,F)=0$.
	\item
		The flow $\Theta$ has exponential convergence rate in a
	neighborhood of every regular point of the vanishing locus of
	$\phi:\scrF_0(\scrT)\to\RR$.
\end{enumerate}
\end{theointro}
The flow $\Theta$ on $\scrF_0$ can be lifted as a flow $\hat\Theta$ on
$\scrM(\scrT)$ and we have the following corollary:
\begin{corintro}
	\label{cor:main}
	There exists a unique map 
	$$\hat
	\Theta:[0,+\infty]\times\scrM(\scrT)\to\scrM(\scrT)$$
	such that
	\begin{enumerate}
		\item
	$d\circ \hat\Theta_t =\Theta_t \circ d$ and
\item $\hat\Theta_0
	=\id$ on $\scrM(\scrT)$,
	\end{enumerate}
	where we used the notation $\hat\Theta_t=\hat\Theta(t,\cdot)$. 
	The map $\hat\Theta$ is the flow of the evolution equation
	given by Formula~\eqref{eq:liftflow} and defines a
	strong deformation
	 retraction of $\scrM(\scrT)$ onto the subspace of 
	polyhedral isotropic maps.
\end{corintro}

The existence of regular points (cf. Definition~\ref{dfn:regular}) of the vanishing set of
$\phi:\scrF_0(\scrT)\to\RR$ is not well understood in comparison to the smooth
setting (cf. \S\ref{sec:mb}). However, we show that the
approximating scheme of \cite{JRT,Rol22} provides examples of regular
points: 
\begin{theointro}
	\label{theo:mainbis}
	In the case where $\Sigma$ is a $2$-torus and
	$f:\Sigma\looparrowright V$ is a smooth isotropic immersion, 
	let $f_N\in\scrM(\scrT_N)$ be the Jauberteau-Rollin-Tapie isotropic polyhedral
	immersions
	of Theorem~\ref{theo:approx} approximating~$f$.
	Then $F_N=df_N \in\scrF_0(\scrT_N)$ 
	is a regular point of the vanishing
	locus of $\phi:\scrF_0(\scrT_N)\to\RR$ for every sufficiently large $N$.
\end{theointro}

\subsection{Open problems and future research}
In view of the approximation scheme given by Theorem~\ref{theo:approx}, it
seems sensible to expect that some type Gromov-Lees 
theorem should apply to the polyhedral setting, provided some flexibility
for the choice of triangulation $\scrT$ of $\Sigma$.
If this is indeed the case, we expect the following consequence:
\begin{conj}
	\label{conj:intro}	
	Let $\Sigma$ be a  surface endowed  with a triangulation
	$\scrT$. Then, up to passing to a subdivision of $\scrT$, 
	the space of polyhedral isotropic immersions in $\scrM(\scrT)$ is
	homotopically equivalent to the space of smooth isotropic
	immersions of $\Sigma$ in $V$.
\end{conj}
For triangulation with lower complexity, an adapted Morse-Bott cohomology
theory for
$\phi:\scrF_0(\scrT)\to\RR$
should lead to topological invariants  for
the space of non constant polyhedral isotropic maps.
Theorem~\ref{theo:main} shows that the polyhedral modified moment map flow
is extremely well behaved, which is a strong incentive to develop a
Morse-Bott theory, with a renormalized flow briefly discussed at
\S\ref{sec:topo}.

A numerical version
of the polyhedral modified moment map flow is currently being developed 
for testing purposes of the Morse-Bott theory and for producing effective
examples of polyhedral isotropic immersed surfaces. 
We sketch the relevant mathematical ingredients of a computer program
that produces approximate solutions of the flow at
\S\ref{sec:numflow}. The code is to be released very soon~\cite{JR1}.

\subsection{Acknowledgments}
A part of this research was carried out in 2023-2024, thanks to the support of
the CNRS International Research Lab
CRM in Montreal. Yann Rollin was hosted by the  CIRGET at UQAM 
and wishes to
thank  all his colleagues  for creating such an enjoyable and 
stimulating mathematical
environment, in particular, Antonio Alfieri, Vestislav Apostolov,
Steven Boyer, Charles Cifarelli, Olivier Collin,  Alexandra
Haedrich, Julien Keller, Abdellah Lahdili, Steven Lu, Duncan McCoy, Frédéric Rochon and Carlo Scarpa.
We  also thank Mélanie Bertelson and Vincent Borrelli, for some useful
discussions.

\tableofcontents

\section{Kähler moment map}
\label{sec:dream}
This section is devoted to the description of an infinite dimensional
moment map geometry, which provides an interpretation of smooth isotropic maps
as some particular zeroes of a moment map.
\subsection{Target space}
\label{sec:target}
Let $V$ be an \emph{affine space} and $\vecV$, its underlying vector space.
We assume that $\vecV$ is a \emph{complex vector space}, of complex dimension
$m\geq 2$, endowed with
a \emph{Hermitian inner product} $h_V$, anti-$\CC$-linear in the
first variable. 
For every $v_1, v_2\in\vecV$, the decomposition
$$
h_V(v_1,v_2) = g_V(v_1,v_2) + i\omega_V(v_1,v_2),
$$
into the real and imaginary parts of $h_V$ provides a \emph{Euclidean inner
product} $g_V$ and a \emph{symplectic form} $\omega_V$.
Multiplication by $i$ defines a linear endomorphism
$i:\vecV\to\vecV$, also called an
\emph{almost complex structure}.
By definition the almost complex structure $i$ is \emph{compatible}
with  $g_V$ and $\omega_V$,
 in the sense
that  
\begin{equation*}
	g_V(iv_1,iv_2)=g_V(v_2,v_2) \quad \mbox{ and } \quad
	 \omega_V(v_1,v_2)=g_V(iv_1,v_2)
\end{equation*}
for every $v_1,v_2\in \vecV$. 

\subsection{Source space}
Let $\Sigma$ be a smooth closed and oriented surface, endowed with 
a  Riemannian metric~$g_\Sigma$. We also assume that $\Sigma$ is connected
in all this paper, for simplicity of notations.
We denote by $\omega_\Sigma$ the \emph{volume
form} of $g_\Sigma$ compatible with the orientation.
The corresponding \emph{almost complex
structure}
$J_\Sigma\in\End(T\Sigma)$ is defined  as a fiberwise rotation of
$T\Sigma\to\Sigma$ with angle
$+\frac\pi 2$ according  to the orientation. 
By construction, $J_\Sigma$ is compatible with $g_\Sigma$ and
$\omega_\Sigma$ is the associated \emph{Kähler form}, in the sense that
$$g_\Sigma(J_\Sigma
\eta_1,J_\Sigma \eta_2)=g_\Sigma(\eta_1,\eta_2) \quad\mbox{ and }\quad
\omega_\Sigma(\eta_1,\eta_2) =
g_\Sigma(J_\Sigma \eta_1,\eta_2),$$
for every 
$x\in\Sigma$ and $\eta_1,\eta_2\in T_x \Sigma$.
In conclusion, $\Sigma$ is endowed with a Kähler structure
$$(\Sigma,g_\Sigma,J_\Sigma,\omega_\Sigma).$$

\subsection{Fiberwise structures on tensor bundles}
\label{sec:fiberwise}
The bundle $T^*\Sigma\to\Sigma$ is identified to $T\Sigma\to\Sigma$ using
the duality induced by the Riemannian metric $g_\Sigma$.
Thus $T^*\Sigma\to\Sigma$ and all the tensor bundles 
are endowed with an induced fiberwise Euclidean inner
product
denoted~$g_\Sigma$ as well. 

The structure of complex vector space on $\vecV$ induces a canonical structure of
complex vector bundle on 
$\Lambda^n\Sigma\otimes
\vecV\to\Sigma$, where the tensor product is taken with respect to $\RR$.
Furthermore, $h_V$ and $g_\Sigma$ induce a fiberwise
Hermitian inner product  on $\Lambda^n\Sigma\otimes\vecV\to\Sigma$, denoted
$h$, and
defined by the following property: for every $x\in\Sigma$, $\beta_1, \beta_2 \in
\Lambda^n_x\Sigma$ and $v_1,v_2\in \vecV$, then
$$
h(\beta_1 \otimes v_1, \beta_2\otimes v_2) = g_\Sigma(\beta_1,\beta_2)
	h_V(v_1,v_2).
$$
The complex structure $J_\Sigma$ acts on $\Lambda^n\Sigma\to\Sigma$, by
composition on the right: for every $x\in \Sigma$, $\beta \in
\Lambda_x^n\Sigma$ and $\eta_1,\cdots,\eta_n\in T_x\Sigma$, we put
$$(\beta\circ J_\Sigma)
(\eta_1,\cdots,\eta_n)=\beta(J_\Sigma\eta_1,\cdots,J_\Sigma\eta_n).
$$
This action is isometric, in the sense that for every
$\beta_1,\beta_2\in\Lambda_x^n\Sigma$, we have
$g_\Sigma(\beta_1,\beta_2)= 
g_\Sigma(\beta_1\circ J_\Sigma,\beta_2\circ J_\Sigma)$

The action of $J_\Sigma$ extends canonically  to
$\Lambda^n\Sigma\otimes\vecV\to\Sigma$ and 
the Hermitian product $h$ is $J_\Sigma$-invariant, in the
sense that, for every $\beta_1,\beta_2\in\Lambda^n_x\Sigma \otimes\vecV$, we
have
$$
h(\beta_1,\beta_2) = h(\beta_1\circ J_\Sigma, \beta_2\circ J_\Sigma).
$$
We consider the Riemannian metric  $g$ given by the real part of $h$.
Then~$J_\Sigma$ and i act isometrically on
$\Lambda^n\Sigma\otimes\vecV\to\Sigma$,
 in the sense that
$$
g(\beta_1,\beta_2)= g(\beta_1\circ J_\Sigma,\beta_2\circ J_\Sigma),\quad
g(i\beta_1,i\beta_2)= g(\beta_1,\beta_2).
$$
We define a fiberwise almost complex structure on $T^*\Sigma\otimes \vecV\to
\Sigma$, by the
formula
\begin{equation}\label{eq:JF}
J\cdot F =- F\circ J_\Sigma.
\end{equation}
By definition, for every $F_1,F_2\in T^*_x\Sigma\otimes \vecV$, we have
$$
g(iF_1,iF_2)=g(J\cdot F_1,J\cdot F_2)=g(F_1,F_2).
$$
In particular, the formula
\begin{equation}\label{eq:omegaF}
\omega(F_1,F_2) = g(J\cdot F_1,F_2)
\end{equation}
defines a fiberwise symplectic form on the bundle $T^*\Sigma\otimes \vecV
\to \Sigma$.

\subsection{Moduli spaces and differentials}
We consider the moduli space of smooth $\vecV$-valued differential $1$-forms
$$
\scrF  = \Omega^1(\Sigma, \vecV)
$$
and the moduli space of smooth maps
$$
\scrM = C^\infty(\Sigma, V).
$$
Notice that 
$\scrM$ is an affine space with $\Omega^0(\Sigma,\vecV)$ as  the underlying
vector space.
For every smooth map, $f:M\to V$, the tangent map $f_*:T\Sigma\to
TV=V\times\vecV$ can
be regarded as a differential 
$$
d f:T\Sigma\to\vecV
$$
given by $d f=\pi_2\circ f_*$, where $\pi_2:TV\to\vecV$ is the second
canonical projection. Hence we have a differential operator between moduli spaces 
$$
\scrM \stackrel {d}\longrightarrow \scrF.
$$
The image of $d$ is the subspace of \emph{exact} $\vecV$-valued  differential
forms, denoted $\scrF_0\subset\scrF$.
Furthermore, $d$ is injective up to  translations by
constant map in $\Omega^0(\Sigma,\vecV)$, identified to $\vecV$.
Thus, $d$ induces a bijection
$$
\scrM/\vecV \stackrel d \longrightarrow  \scrF_0 .
$$
\subsection{Euclidean structure and Hodge theory}
The fiberwise Euclidean inner product $g$ on the vector bundle $\Lambda^n
\Sigma \otimes
\vecV\to \Sigma$
induces an $L^2$-Euclidean inner product on $\Omega^n(\Sigma, \vecV)$, given by
\begin{equation}
\label{eq:dfnGmetric}
\cG(\beta_1,\beta_2) = \int_\Sigma g(\beta_1,\beta_2)\;\omega_\Sigma, \quad\mbox{for
	$\beta_1,\beta_2 \in \Omega^n(\Sigma, \vecV)$.}
\end{equation}
For simplicity, we will often use the notations
$$
\ip{\beta_1,\beta_2} = g(\beta_1,\beta_2),\quad 
\ipp{\beta_1,\beta_2} = \cG(\beta_1,\beta_2)
$$
and
$$
|\beta|=\sqrt{g(\beta,\beta)},\quad \|\beta\|_{L^2}=
\sqrt{\cG(\beta,\beta)}.
$$
 The formal adjoint $d^\star$ of $d$ is defined by the property
$$
\cG(\beta_1,d\beta_2)=\cG(d^\star\beta_1,\beta_2) \quad\mbox {for
$\beta_1\in \Omega^{n+1}(\Sigma,
\vecV)$
and $\beta_2 \in \Omega^{n}(\Sigma, \vecV)$}
$$
 and the \emph{Laplace operator} 
$\Delta$ is given by the formula
 $$\Delta=dd^\star+d^\star d.$$
The classical \emph{Hodge theory} extends to $\vecV$-valued differential
forms.
In particular, we obtain  a $\cG$-orthogonal projection onto
the space of $\vecV$-valued exact $1$-forms denoted
$$
\Pi :\scrF\to\scrF_0.
$$

\emph{Hölder spaces} are better suited for elliptic operators, as the Laplace operator $\Delta$.
Recall that the Riemannian metric $g_\Sigma$ and the fiberwise inner product $g$
induce a $C^{k,\nu}$-\emph{Hölder norm} for tensor fields over $\Sigma$,
denoted $\|\cdot\|_{k,\nu}$,
where $k\in\NN$ is the number of derivatives and $\nu\in(0,1)$ is the Hölder
regularity exponent for the $k$-th derivative~(cf. \cite{Joy} for an
explicit definition).

Hölder norms define Hölder completed spaces. In particular, we
denote 
by $\scrF^{k,\nu}$ and $\scrM^{k,\nu}$ the completion of $\scrF$ and
$\scrM$.
For every $F\in\scrF^{k,\nu}$  with $k\geq 0$, the \emph{Hodge
decomposition} theorem
states that $F$ admits a $\cG$-orthogonal decomposition
$$
F=F_h +\Delta G,
$$
where $G\in\scrF^{k+2,\nu}$ and $F_h$ is a smooth harmonic form.
In particular, we can define an orthogonal \emph{Hodge projection}
$\Pi:\scrF^{k,\nu}\to\scrF^{k,\nu}_0$ by
$$
\Pi(F)= dd^\star G
$$
of $F$ onto its exact component.
By Proposition~\ref{prop:hodge}, the projector $\Pi$ is continuous with
respect to Hölder topology. This result, which is an immediate consequence of Hodge
theory,  is a crucial argument for the proof of
Theorem~\ref{theo:smoothflow}, via the Cauchy-Lipschitz theorem (cf.
Theorem~\ref{theo:lip}). We provide a proof of the proposition for the sake of
self-containedness.
\begin{prop}
\label{prop:hodge}
	For $k\geq 0$ and $\nu\in(0,1)$, the orthogonal 
	projection $\Pi:\scrF^{k,\nu} \to \scrF_0^{k,\nu}$ onto the
	exact component of a differential form is a continuous linear map with respect to the
$C^{k,\nu}$-norm. 
In other words, there exists a real constant $c>0$ such that
$$
\|\Pi(F)\|_{{k,\nu}}
\leq c\|F\|_{{k,\nu}},\quad \mbox{for every $F\in\scrF^{k,\nu}$.}
$$
\end{prop}
\begin{proof}
	For $F\in\scrF^{k,\nu}$, the Hodge decomposition theorem shows that
	\begin{equation}
		\label{eq:hodge1}
F= F_h + \Delta G,
	\end{equation}
	where $G\in \scrF^{k+2,\nu}$ is a $1$-form orthogonal to harmonic forms and $F_h$ is the
	harmonic part of $F$.

	A $C^{k,\nu}$-estimate for $F$
provides a	control on the $C^{k-2,\nu}$-norm of $\Delta F$.
	By Formula~\eqref{eq:hodge1}, we have $\Delta F = \Delta^2 G$, since $F_h$ is harmonic.
	Hence the $C^{k,\nu}$-norm of
	$F$ control the $C^{k-2,\nu}$-norm of $\Delta^2 G$. The operator
	$\Delta$ is selfajoint, hence $\Delta G$
	is orthogonal to the kernel of $\Delta$. Then, elliptic Schauder estimates
	provide a control on the $C^{k,\nu}$-norm of $\Delta G$. Since $G$ was chosen
	orthogonal to harmonic forms, we deduce a $C^{k+2,\alpha}$ control
	on $G$ by the Schauder estimates.
	In conclusion, there exists a universal constant $c_1>0$, such that
$$
	\|G\|_{{k+2,\nu}}\leq c_1\|F\|_{{k,\nu}}.
$$
	 Finally, the $C^{k+2,\nu}$-norm of $G$ controls the
	$C^{k,\nu}$-norm of $dd^\star G$ and we 
	conclude that there exists a universal constant $c>0$ that
	satisfies the proposition.
\end{proof}

\subsection{Kähler structures on the moduli space}
The space of differential forms $\Omega^n(\Sigma, \vecV)$ has a structure
of module over $C^\infty(\Sigma,\CC)$, acting by complex multiplication on
$\vecV$-valued forms. More precisely
$$
(\lambda F)_x = \lambda(x)F_x
$$
for every $\lambda \in C^\infty(\Sigma,\CC)$, $F\in\Omega^n(\Sigma,\vecV)$ and
$x\in\Sigma$.

The space  of differential $1$-forms $\scrF$ admits an alternate almost
complex structure 
$$\cJ:\scrF\to\scrF,
$$
defined by
\begin{equation}
	\label{eq:dfnJ}
	(\cJ \cdot F)_x\cdot \eta  =  (JF_x)\cdot \eta = - F_x\circ J_\Sigma \cdot \eta,
\end{equation}
where $J$ is defined by Formula~\eqref{eq:JF},
for every $x\in\Sigma$ and $\eta\in T_x\Sigma$.
By construction, the almost complex structure $\cJ$ is compatible with the
Eulidean $L^2$-inner product $\cG$ on $\scrF$. The corresponding
 Kähler form $\Omega$ deduced from $\cG$ and $\cJ$ is given by
$$\Omega(\dot F_1,\dot F_2)= \cG(\cJ\dot F_1,F_2)$$
for every $\dot F_1,\dot F_2\in\scrF$.
Equivalently
$$
\Omega(\dot F_1,\dot F_2) = \int_\Sigma \omega(\dot F_1,\dot
F_2)\omega_\Sigma,
$$
where $\omega (\dot F_1,\dot F_2)= g(J\cdot \dot F_1,\dot F_2)$.

In conclusion, we have a natural  Kähler
structure 
$$
(\scrF,\cG,\cJ,\Omega)
$$
on the moduli space $\scrF$.

\subsection{An involution}
The complex vector space $\scrF$ 
 admits several almost complex
structures: the almost complex structure $i:\scrF\to\scrF$, corresponding
to
the multiplication by $i$ and the almost complex structure $\cJ$ described above.
By construction, the two almost complex structure $i$ and $\cJ$ commute.
Therefore,  the endomorphism 
$$
\scrR:\scrF\to\scrF
$$
defined by
$$
\scrR F =  i\cJ \cdot F
$$
is a \emph{linear isometric involution}  of $\scrF$, which commutes with $i$
and $\cJ$. We obtain a $\cG$-orthogonal decomposition
$$
\scrF=\scrF^+ \oplus \scrF^-
$$
where $\scrF^\pm$ are the eigenspaces associated to the eigenvalues $\pm 1$ of
$\scrR$. More explicitely, we have
$$
\scrF^+=\{F\in\scrF, F\circ J_\Sigma = i F \} \quad \mbox{ and }
$$
$$\scrF^-=\{F\in\scrF, F\circ J_\Sigma = - i F \}.
$$ 
In other words,  $\scrF^+$ (resp. $\scrF^-$)  is the
subspace of
complex  (resp. anti-complex) morphisms  $F:T\Sigma\to \vecV$.  
Hence, every $F\in\scrF$ admits a unique orthogonal decomposition
$$
F=F^++F^-,
$$
where $F^\pm\in\scrF^\pm$. By definition, we have
$$
\scrR F= F^+-F^-.
$$
\subsection{Gauge group action}
We  consider the infinite dimensional complex Lie group 
$$
\TT^\CC = C^\infty(\Sigma,\CC^*),
$$
with trivial Lie algebra 
$$\LieTT^\CC= C^\infty(\Sigma,\CC).$$
We define an action of $\TT^\CC$ on $\scrF$ as follows:
given $\lambda\in\TT^\CC$ and
$F\in\scrF$,
we put
\begin{equation}
	\label{eq:cxaction}
\lambda \cdot F = \bar\lambda^{-1} F^+ + \lambda F^-,
\end{equation}
where $F=F^++F^-$ according to the splitting $\scrF=\scrF^+\oplus \scrF^-$
and $\bar\lambda$ denotes the complex conjugate.
By construction the action of $\TT^\CC$ preserves the almost complex
structure $\cJ$ for the following obvious reason:
\begin{lemma}
Each
	subspace $\scrF^\pm$ is $\cJ$ invariant and the action of $\TT^\CC$ on $\scrF$ is $\cJ$-linear. 
\end{lemma}

The group $\TT^\CC$ contains the real subgroup
$$
\TT=  C^\infty(\Sigma,S^1),
$$
where $S^1\subset \CC$ is the unit circle.
	If $\lambda\in \TT$, then
$\bar\lambda^{-1} = \lambda$, and group action is merely  given by complex
multiplication:
$$
\lambda\cdot F= \lambda F,\quad \mbox{ for every } \lambda \in \TT.
$$

\subsection{Infinitesimal gauge group action}
For $\zeta\in \LieTT^\CC=C^\infty(\Sigma,\CC)$  we define an exponential map
 $$\exp :\LieTT^\CC \to \TT^\CC,
 $$
 by 
 $$\exp(\zeta)=e^{i\zeta},$$
 so that the space of real valued functions $\LieTT=C^\infty(\Sigma,\RR)$ is identified to the Lie
 algebra of $\TT$. 
As usual, the infinitesimal action of $\zeta\in\LieTT^\CC$ is the vector
field $X_\zeta$ on the moduli space $\scrF$, defined by
$$
 X_\zeta(F) = \left . \frac d{dt}\right |_{t=0}\exp (t\zeta)\cdot F. 
$$
 Formula~\eqref{eq:cxaction} shows that 
\begin{equation}
	\label{eq:infcxaction}
 X_\zeta(F) =i\bar\zeta F^+ + i\zeta F^-.
\end{equation}
In particular, if $\zeta\in\LieTT$,  we have
\begin{equation}
	\label{eq:infcxactionreal}
X_\zeta(F)= i\zeta F.
\end{equation}
If, on the contrary, $\zeta\in i\LieTT$ is a purely imaginary function, we have
\begin{equation}
	\label{eq:infcxactionim}
X_\zeta(F)= -i\zeta \scrR F.
\end{equation}
By~\eqref{eq:infcxactionreal} and \eqref{eq:infcxactionim}, the
infinitesimal action of $\zeta \in
\LieTT$ satisfies
$$
\cJ \cdot X_{\zeta}(F) = \cJ \cdot  i\zeta F = \zeta \scrR F  = -i (i\zeta) \scrR F =
 X_{i\zeta}(F),
$$
and we deduce the following result:
\begin{lemma}
	The action of $\TT^\CC$ on $\scrF$ is
	 the $\cJ$-complexification of the action of $\TT$. 
\end{lemma}

In addition, the gauge group $\TT$ acts isometrically on $\scrF$. Indeed for every
$\lambda\in \TT$ and $F\in \scrF$,
$$
\|\lambda\cdot F \|^2_{L^2} = \|\lambda F\|^2_{L^2} =
 \|F\|^2_{L^2}.
$$
In conclusion, we have the following result:
\begin{prop}
	\label{lemma:inftac}
	 The  action
	of $\TT$ on $\scrF$ is linear and preserves the Kähler structure
	$(\scrF,\cG,\cJ,\Omega)$. 
\end{prop}

\subsection{Symplectic density}
\label{sec:hamtorus}
The  map
\begin{equation}
	\begin{tikzcd}[  baseline=(current  bounding  box.center), cramped,
  row sep = 0ex,
  column sep = 1.5em,
  /tikz/column 1/.append style={anchor=base east},
  /tikz/column 2/.append style={anchor=base west}
   ]
		\mu :  \scrF \ar[r] & \LieTT =C^\infty(\Sigma,\RR)  \\
		 F\ar[r,mapsto] & 
		- \frac 12 g(F,\scrR F).
\end{tikzcd}
\end{equation}
 can be interpreted
as a \emph{symplectic density}, according to the following lemma:
\begin{lemma}
	\label{lemma:mm}
	The following formulas hold 
\begin{equation}
	\label{eq:mm}
\mu(F)
= -\frac 12 \omega(iF, F) = -\frac
12(|F^+|^2 - |F^-|^2)= -\frac{F^*\omega_V}{\omega_\Sigma},
\end{equation}
	for every $F\in\scrF$, where the pullback is defined by
	$(F^*\omega_V)(\eta_1,\eta_2) =
	\omega_V(F(\eta_1),F(\eta_2))$, for every $x\in\Sigma$ and
	$\eta_1,\eta_2\in
	T_x\Sigma$.
\end{lemma}
\begin{proof}
	Using the fact that $-\frac 12g(\scrR F, F)=-\frac 12 g(i\cJ F,F) 
	= -\frac 12 g(\cJ i F, F) =-\frac 12
	\omega (iF,F)$ we deduce the first identity.
	
	Using the  decomposition $F= F^+ +F^-$, we write
	$g(\scrR F,F)= g(F^+ - F^- , F^+ + F^-)$. Using the fact the $F^+$
	and $F^-$ are pointwise $g$-orthogonal, we deduce that $g(\scrR F,
	F) = |F^+|^2- |F^-|^2$,
	which proves the second identity.

Let $x\in\Sigma$ and $(e_1,e_2)$ be a $g_\Sigma$-orthonormal oriented basis of
	$T_x \Sigma$. In particular $\omega_\Sigma(e_1,e_2)=1$ and
	$J_\Sigma e_1 =e_2$.
	Hence $(\cJ\cdot F)(e_1) = -F(J_\Sigma e_1) = -F(e_2)$. Similarly
	$(\cJ\cdot F)(e_2) = F(e_1)$. Hence $\scrR F(e_1)= -iF(e_2)$
	and $\scrR F(e_2)=iF(e_1)$. By definition
	\begin{align*}
		\mu(F)(x) &= -\frac 12 g(\scrR F, F)(x) \\
		&=- \frac 12 g_V(\scrR F(e_1), F(e_1)) - \frac 12
		g_V(\scrR F(e_2), F(e_2)) \\
		&= -\frac 12( - g_V(iF(e_2),F(e_1)) + g_V(iF(e_1) ,
		F(e_2)) )\\
		&= -\omega_V(F(e_1),F(e_2)) \\
		&=-
		\frac{(F^*\omega_V)(e_1,e_2)}{\omega_\Sigma(e_1,e_2)},
	\end{align*}
which proves the last identity of the lemma.
\end{proof}

We introduce some h-principle terminology before stating an immediate
corollary below.
\begin{dfn}
	\label{dfn:monoiso}
A differential form $\scrF$ is called isotropic if $F:T\Sigma\to \vecV$ maps
	every tangent space to an isotropic subspace of $V$ or,
	equivalently, if $F^*\omega_V=0$. If the restriction of
	$F:T\Sigma\to \vecV$ to every tangent space is injective, we say
	that $F$ is a monomorphism.
\end{dfn}

\begin{cor}
	\label{lemma:vanishmm}
	Let $F$ be an element of $\scrF$.
	The following properties are equivalent:
	\begin{enumerate}
		\item \label{it:isotropic} $F\in\scrF$ is isotropic
	\item  \label{it:zero}$\mu(F)=0 \in \LieTT$ .
		\item $|F^+| = |F^-|$ identically on $\Sigma$.
	\end{enumerate}
\end{cor}
If $F\in\scrF_0$, we deduce the following interpretation for isotropic
maps:
\begin{cor}
	\label{cor:Atech}
For every map $f\in\scrM$, the following properties are equivalent:
	\begin{enumerate}
		\item $f$ is an isotropic map.
		\item $F=df\in\scrF_0$ is isotropic.
		\item The map $f$ satisfies the equation $\mu(d f)=0$.
	\end{enumerate}
	In particular, the space of isotropic maps modulo $\vecV$ agrees with 
	the zeroes of $\mu$ in $\scrF_0$ via the bijection
	$d:\scrM/\vecV\to\scrF_0$.
\end{cor}

\subsection{Hamiltonian action}
We now show that $\mu$ is indeed a moment map:
\begin{theo}
	\label{theo:mm}
The action of $\TT$ on $\scrF$ is Hamiltonian with
	moment map~$\mu$. More precisely, $\mu$ is $\TT$-invariant and
$$
	D\ipp{\mu,\zeta}=- \iota_{X_\zeta}\Omega
$$
for every $\zeta \in\LieTT$.
\end{theo}
\begin{proof}
	The invariance of $\mu$ is clear, by definition.
The proof of the theorem starts with a computation:
	\begin{lemma}
		\label{lemma:dmu}
For every $F,\dot F\in\scrF$, we have 
		$$
		D\mu|_F\cdot \dot F =- g(\scrR F,\dot F).
$$
	\end{lemma}
	\begin{proof}
By bilinearity, $2D\mu|_F \cdot F= - g(\scrR
		F,\dot F) -g(F,\scrR\dot F)$ and he lemma
		follows from the fact that $\scrR$ is pointwise
	$g$-selfadjoint.
	\end{proof}
	For $\zeta\in\LieTT$, we have $X_\zeta(F)=i\zeta F$ by
	Formula~\eqref{eq:infcxactionreal} and it follows that 
	\begin{align*}
		\Omega(X_\zeta(F),\dot F) &= \Omega(i\zeta  F,\dot F ) \\
		&= \ipp{
			\cJ i\zeta F,\dot F}\\
		&= \ipp{\zeta \scrR F,\dot F}\\
		&=  \int_\Sigma g(\zeta \scrR
	F,\dot F) \omega_\Sigma \\
		&= -\int_\Sigma 	\zeta D\mu|_F\cdot \dot F 
	\omega_\Sigma\\
		&=-\ipp{ D\mu|_F\cdot\dot F,\zeta},
	\end{align*}
	 which proves
	the theorem.
\end{proof}
\begin{proof}[Proof of Theorem~\ref{theo:A} and Corollary~\ref{cor:A}] 
The restatement of the constructions carried out at
	\S\ref{sec:dream}, together with Theorem~\ref{theo:mm} and
	Lemma~\ref{lemma:mm} prove Theorem~\ref{theo:A}.
	Corollary~\ref{cor:A} is a restatement of Corollary~\ref{cor:Atech}.
\end{proof}
\subsection{Stability and isotropic maps}
\label{sec:kn}
The Kempf-Ness theory, for Kähler moment map geometry, relates the symplectic
 reduction with geometric invariant theory. This point of
view seems appealing in our case, where $\scrF$ is acted on by the complex
gauge group $\TT^\CC$ and the $\mu$ is a moment map for $\TT$.
	The question of existence of a zero of the moment map in a
	$\TT^\CC$-orbit
is rather trivial:
for simplicity,
we define the $\TT^\CC$-invariant subspace of \emph{generic} differentials
$\scrF_{gen}$ as the subspace of nowhere vanishing differential forms.
Then we have the following result
\begin{lemma}
	For every $F\in\scrF_{gen}$, the following
	properties are equivalent:
	\begin{enumerate}
		\item $F^+\in\scrF_{gen}$ and $F^-\in\scrF_{gen}$.
		\item There exists $\lambda\in\TT^\CC$ such that
			$\mu(\lambda\cdot F)=0$.
	\end{enumerate}
In particular, the orbit of $F\in \scrF^\pm\setminus 0$ does not
contain any zero of the moment map.
\end{lemma}
\begin{proof}
For $\lambda:\Sigma\to \RR\setminus 0$ we have
$$
\mu(\lambda\cdot F)= -\frac 12 (\lambda^{-2}|F^+|^2-\lambda^2|F^-|^2)
$$
and the lemma is obvious.
\end{proof}

However, we are mostly interested in the zeroes of the moment map in
$\scrF_0$, as they are differentials of isotropic maps by
Corollary~\ref{cor:A}.
It would be
interesting to understand how the space of isotropic maps in $\scrM$ is
related to some notion of geometric stability on $\scrF$. 
Unfortunately, the image $\scrF_0=d(\scrM)$  is not $\TT$-invariant and it is not clear how
to obtain an analogue of Kempf-Ness theory from this point.

\subsection{Energy of the moment map}
We consider the energy of the moment map $\mu$, given by the functional
\begin{equation}
	\begin{tikzcd}[  baseline=(current  bounding  box.center), cramped,
  row sep = 0ex,
  column sep = 1.5em,
  /tikz/column 1/.append style={anchor=base east},
  /tikz/column 2/.append style={anchor=base west}
   ]
		\phi :  \scrF \ar[r] & \RR  \\
		 F\ar[r,mapsto] & 
		\phi(F)= \frac 12\|\mu(F)\|^2_{L^2}
\end{tikzcd}
\end{equation}
Obviously, $\phi$ is non negative and 
$$
\phi^{-1}(0)=\mu^{-1}(0)
$$
which is to say that the vanishing locus is the space of isotropic
differential forms in $\scrF$. 
By Corollary~\ref{cor:A}, 
the vanishing locus of $\phi:\scrF_0\to\RR$ is identified to the subspace of
isotropic maps in $\scrM$ modulo the action of $\vecV$ by translations via
the correspondence $d:\scrM/\vecV\to\scrF_0$.

We prove various formulas about the differential and the gradient of the
functional $\phi$ on $\scrF$: 
\begin{prop}
	\label{prop:gradient}
	For every $F,\dot F\in\scrF$, we have
	$$D\phi|_F\cdot \dot F=-\ipp{\mu(F),g(\scrR F,\dot F)},$$
	$$
	D\phi|_F\cdot F= 4\phi(F),
	$$
	and the gradient of the functional $\phi:\scrF\to \RR$ 
	is given by the formula, 
$$
	\nabla\phi(F)=-\mu(F)\scrR F.
$$
or 
	\begin{equation}
\nabla \phi  = -\cJ Z,
	\end{equation}
where $Z$ is the vector field on $\scrF$ defined by
$$
	Z(F) = X_{\mu(F)}(F).
	$$
\end{prop}
\begin{proof}
	By definition of $\phi(F)=\frac 12\|\mu(F)\|^2$ hence 
	$$
	\left . D\phi\right |_F \cdot \dot F=  \ipp{\mu(F),D\mu|_F
	\cdot \dot
	F  }.
	$$
	By Lemma~\ref{lemma:dmu}, we have
	$$
	\left . D\phi\right |_F \cdot \dot F= - \ipp{ \mu(F), g(\scrR
	F,\dot F) }.
	$$
	In particular, for $\dot F=F$, we have
	$	D\phi|F\cdot F = -\ipp{\mu(F),g(\scrR F,F)} =
	2\ipp{\mu(F),\mu(F)}=2\|\mu(F)\|^2=4\phi(F)$.

	Now 
	\begin{align*}
		D\phi|_F\cdot \dot F&=	\ipp{\mu(F),g(\scrR F,\dot F)} \\
		&= \int _ \Sigma \mu(F)g(\scrR F,\dot
	F) \omega_\Sigma \\
		&= \int _ \Sigma g(\mu(F)\scrR F,\dot
	F) \omega_\Sigma \\
		&=  \ipp{\mu(F)\scrR F , \dot F}
	\end{align*}
	and we deduce that 
	$$	\nabla \phi (F)= -\mu(F)\scrR F.$$

	In particular $\cJ\nabla\phi(F)= -\mu(F)\cJ\scrR F = \mu(F)i F =
	X_{\mu(F)}(F)$ and we conclude that
$$
\nabla\phi =- \cJ Z
$$
	where $Z$ is the vector field on $\scrF$ defined by
	$Z(F)=X_{\mu(F)}(F)$.
\end{proof}
\begin{cor}
	\label{cor:crit}
	The following properties are equivalent for $F\in\scrF_0$ 
	\begin{enumerate}
		\item $F$ is a zero of $\phi:\scrF_0\to\RR$. 
		\item $F$ is a zero of $\mu$ in $\scrF_0$.
		\item $F$ is isotropic.
		\item $F$ is a critical point of $\phi:\scrF_0\to\RR$.
	\end{enumerate}
\end{cor}
\begin{proof}
	We know that $(2)\Leftrightarrow (3)$ by Corollary~\ref{cor:Atech} and
	we prove that $(1)\Rightarrow(2)\Rightarrow(4)\Rightarrow
(1)$.
If $F$ is a zero of $\phi:\scrF_0\to\RR$, it is a zero of $\mu$, by
	definition of $\phi$. If $\mu(F)=0$, then $D\phi|_F=0$ by
	Proposition~\ref{prop:gradient}, hence $F$ is a critical point of
	$\phi:\scrF\to\RR$. This implies that $F$ is also a
	critical point of the restricted functional $\phi:\scrF_0\to\RR$.
	If $F$ is a critical point of the restriction, then $D\phi|_F\cdot
	F= 0$. By Proposition~\ref{prop:gradient}, we deduce that
	 $0=D\phi|_F\cdot F=4\phi(F)$ and 
	we conclude that $\phi(F)=0$.
\end{proof}

\section{The modified moment map flow}
\label{sec:smoothflow}
The subspace of exact differentials $\scrF_0\subset\scrF$ 
is not stable under the $\TT$-action. 
Hence, the gradient $\nabla\phi$ is generally not tangent to
	$\scrF_0$.
The restriction of the functional $\phi:\scrF_0\to \RR$
admits a gradient vector field  $\nabla^\circ \phi$  related to
$\nabla\phi$ by the formula
$$
\nabla^o\phi(F) = \Pi (\nabla \phi(F)).
$$
We define the \emph{modified moment map flow} by the evolution equation
\begin{equation}
	\label{eq:mmmf}
	\frac{\del F}{\del t} =-\nabla^o\phi(F)
\end{equation}
along $\scrF_0$ and, more generally, along the Hölder completion
$\scrF_0^{k,\nu}$.

\subsection{Basic properties of the flow}
An interesting feature of the flow is the following decay
property:
\begin{theo}
	\label{theo:flownorm}
	Let $F_t $ be a solution of the modified moment map
	flow, for $t$ in some interval of $I$. Then 
	\begin{equation}
		\label{eq:flownorm}
		\frac \del{\del t} \|F_t\|^2_{L^2} =-8\phi(F_t). 
	\end{equation}
	In particular $t\mapsto \|F_t\|_{L^2}$ is non increasing along $I$.
\end{theo}
\begin{proof}
We have 
	\begin{align*}
		\frac 12	
		\frac \del{\del t} \|F_t\|_{L^2}^2 
		&= \ipp{\frac{\del F_t}{\del
		t},F_t }\\
		&=
	-\ipp{\nabla^o\phi(F_t),F_t}\\
		&= -\ipp{\nabla\phi(F_t),F_t}\\
		&= -4\phi(F_t).
	\end{align*}
\end{proof}
Another important essential property is that the stationary points of the modified
moment map flow are exactly the zeroes of $\phi$.
\begin{prop}
	\label{prop:fps}
	The fixed point set of the modified moment map flow  agrees with the
	vanishing set of $\phi:\scrF_0\to\RR$.
\end{prop}
\begin{proof}
	This is an immediate consequence of Corollary~\ref{cor:crit}.
\end{proof}
The modified moment map flow has the short time existence property
according to the following theorem:
\begin{theo}
	\label{theo:lip}
	For every $k\geq 0$, 
	$\nu\in(0,1)$ and $F\in \scrF_0^{k,\nu}$, there exists
	$\epsilon >0$ and a unique  
	solution of the modified moment map  flow $F_t \in\scrF_0^{k,\nu}$,
	defined for  $t\in
	[-\epsilon,\epsilon]$ such that  $F_0=F$.
\end{theo}
\begin{proof}
	The map $F\mapsto \nabla\phi(F)$ is polynomial of order $3$ in the
	coefficients of $F$ by Proposition~\ref{prop:gradient}. 
	In particular, it is locally Lipschitz with respect to
	the $C^{k,\nu}$-norm.
	By Proposition~\ref{prop:hodge}, the Hodge projector
	$\Pi :\scrF^{k,\nu}\to\scrF_0^{k,\nu}$
	 is
	continuous for the $C^{k,\nu}$-norm. Therefore
	the maps $F\mapsto \nabla^\circ\phi(F)=\Pi\nabla\phi(F)$ is locally Lipschitz on
	$\scrF_0^{k,\nu}$. The
	Cauchy-Lipschitz theorem applies and the theorem is proved.
\end{proof}
\begin{proof}[Proof of Theorem~\ref{theo:smoothflow}]
	The result is just a restatement of Theorem~\ref{theo:lip},
	Proposition~\ref{prop:fps} and
	Theorem~\ref{theo:flownorm}.
\end{proof}

\subsection{Energy as a Morse-Bott function}
\label{sec:mb}
By Proposition~\ref{prop:gradient}, we have
$$
D\phi|_F\cdot \dot F_1=
-\ipp{\mu(F),g(\scrR F,\dot F_1)}
$$
 for every $F, \dot F_1 \in\scrF$ and
 second variation of $\phi$ with respect to some $\dot
F_2\in\scrF$ is given by
$$
D^2\phi|_F\cdot (\dot F_1,\dot F_2)= \ipp{D\mu|_F\cdot \dot
F_1,D\mu|_F\cdot \dot F_2} - \ipp{\mu(F), g(\scrR \dot F_2,\dot F_1)}.
$$
In particular, if $\mu(F)=0$, we have
\begin{equation}
	\label{eq:hesssmooth}
D^2\phi|_F\cdot (\dot F_1,\dot F_2)= \ipp{D\mu|_F\cdot \dot
F_1,D\mu|_F\cdot \dot F_2} 
\end{equation}
which proves the following lemma:
\begin{lemma}Let $F\in\scrF$ be an isotropic differential. Then,  $D^2\phi|_F$ is a non negative bilinear
	form. Furthermore, the kernel of $D^2\phi|_F$ is given 
	by the  kernel of the linear map $\ker D\phi|_F$.
\end{lemma}
According to the above lemma, the functional $\phi:\scrF_0\to\RR$ seems to behave like a
Morse-Bott function. The rest of this section is devoted to  proving that
this is indeed the case,  in a neighborhood
of a monomorphism.

Let $s_0$ be the $0$-section of the bundle $T^* \Sigma\to \Sigma$.
The  map 
$$
\Phi:\Diff_{k,\nu}(\Sigma)\times \Omega^1_{k,\nu}(\Sigma)\to C^{k,\nu}(\Sigma,T^*\Sigma)
$$
defined by $\Phi(\varphi,s)= s \circ \varphi$ for some $k\geq 1$ 
is smooth. Furthermore $\Phi$ defines a local diffeomorphism between a neighborhood of
$(\id,s_0)$ and a neighborhood of $s_0$ in $C^{k,\nu}(\Sigma, T^*\Sigma)$.

It is a well known fact that $T^*\Sigma$ is endowed with a canonical
symplectic form $\omega_{T^*}=d\Lambda$, where $\Lambda$ is the \emph{Liouville
form} on $T^*\Sigma$.
Furthermore, a map $h:\Sigma\to T^*\Sigma$,
sufficiently close to $s_0$ in $C^1$-norm is isotropic if, and only if,
$h=\Phi(\phi,s)$ for some closed differential $1$-form $s$
\cite{MDS}.

It follows that the subspace of isotropic maps is a submanifold of
$C^{k,\nu}(\Sigma, T^*\Sigma)$ in a neighborhood of $s_0$. Furthermore, the
subspace of isotropic maps is locally
diffeomorphic via $\Psi$ to  the submanifold 
 $\Diff_{k,\nu}(\Sigma)\times \Omega^{1,c}_{k,\nu}(\Sigma)$ in a neighborhood of
 $(\id,s_0)$, where $\Omega^{1,c}_{k,\nu}(\Sigma)$  is the subspace of closed forms.

We now specialize to the particular case where $V$ has real dimension $4$.
Let $f:\Sigma\looparrowright V$ be an isotropic immersion. Then, the
image $M=f(\Sigma)$ is known  an \emph{immersed Lagrangian
surface} of $V$.
By the Lagrangian neighborhood theorem~\cite{MDS}, there exists a neighborhood $\scrU$ of
the zero section $s_0$ in $T^*\Sigma$ and a neighborhood $\scrV$ of $M$ in $V$, together with a smooth
immersion $\theta:\scrU\to \scrV$, such that $\theta\circ s_0= f$ and
$\theta^*\omega_V=\omega_{T^*}$.

Every map $\tilde f:\Sigma\to
V$, sufficiently
close to $f$ $C^1$-norm, is given by a map $h:\Sigma\to T^*\Sigma$ 
sufficiently close to $s_0$, such that $\tilde f = \theta\circ h$.
Furthemore, $\tilde f$ is isotropic if, and only if, $h$ is isotropic since
$\theta$ preserves the symplectic forms.
We put $\Psi(\varphi,s)=\theta \circ s \circ\varphi$ and deduce the following result:
\begin{theo}
	We assume that $\dim_\RR V=4$ and let $f:\Sigma\looparrowright V$
	be an isotropic immersion with Hölder
	regularity $C^{k,\nu}$, for some $k\geq 1$.
There exists  smooth map
$$\Psi:\Diff_{k,\nu}(\Sigma)\times \Omega^1_{k,\nu}(\Sigma)\to
	C^{k,\nu}(\Sigma,V)$$
	defined in a neighborhood of $\Diff_{k,\nu}(\Sigma)\times\{s_0\}$,
	such that 
	\begin{enumerate}
		\item $\Psi(\id,s_0)=f$
		\item $\Psi$ is $\Diff_{k,\nu}(\Sigma)$-equivariant
		\item $\Psi$ is a local diffeomorphism in a neighborhood of
			$(\id,s_0)$
		\item The map $\Psi$ restrict as a local diffeomorphism at
			$(\id,s_0)$ between  $\Diff_{k,\nu}(\Sigma)\times
			\Omega^{1,c}_{k,\nu}(\Sigma)$  the isotropic
			deformations of $f$.
	\end{enumerate}
	In particular, the subspace of isotropic deformations in a
	neighborhood  is a
	submanifold of $C^{k,\nu}(\Sigma)$ in a sufficiently small open neighborhood of $f$.
\end{theo}
For $k\geq 1$, the differential $d$ induces a diffeomorphism
$$
\scrM^{k,\alpha}\to\scrF^{k-1,\alpha}_0\times V
$$
given by $f\mapsto (df,f(x_0))$, where $x_0\in\Sigma$ is some fixed marked
point.
The above  diffeormorphism restrics to a diffeomorphism between isotropic maps and
istropic differentials and we deduce the following corollary:
\begin{cor}
	\label{cor:mb}
	Assuming $\dim_\RR V=4$ and $k\geq 0$, let $F\in\scrF_0^{k,\nu}$ be an isotropic
	monomorphism, for some $k\geq 0$.
	Then the subspace of isotropic differentials is a submanifold of
	$\scrF_0^{k,\nu}$ in a neighborhood of
	$F$.
Furthemore, the Hessian of $\phi$ is positive in transverse direction to
	the submanifold of isotropic differential near $F$.
\end{cor}

\section{Polyhedral isotropic surfaces}
\label{sec:poly}

In this section, we adapt all the smooth constructions  of \S\ref{sec:dream} and
\S\ref{sec:smoothflow} to the \emph{polyhedral
setting}. 

\subsection{Definition of polyhedral surfaces}
A  \emph{triangulation} of a surface $\Sigma$ is a triple
$\scrT=(\Sigma,\scrK,\ell)$, where:
\begin{enumerate}
\item $\Sigma$ is a topological surface,
\item $\scrK$ is a locally finite \emph{simplicial complex}, contained in some ambient affine
	space $A$ and
\item $\ell:|\scrK|\to\Sigma$ is a homeomorphism, where $|\scrK|$ is the
	topological space associated to the simplicial complex.
\end{enumerate}

Given a triangulation $\scrT$, every simplex $\sigma\in\scrK$ 
defines a subset $\ell(\sigma)\subset \Sigma$, homeomorphic to
$\sigma$. It is often convenient to think of $\sigma$ as a
domain in $\Sigma$, using the homeomorphism $\ell$ and we will take
the liberty to drop
the reference to~$\ell$ in our notations.

Every  simplex $\sigma$ of an affine space~$A$ spans an affine subspace of
$A$, with underlying vector space
denoted $\vec \sigma$, also called the \emph{tangent direction of the
simplex} $\sigma$.

Let $g_\Sigma= (g_\sigma)_{\sigma\in\scrK}$, be a family,
where $g_\sigma$   is a Euclidean inner product on $\vec\sigma$. 
Suppose that  for
every $\sigma_1,\sigma_2\in\scrK$ with
$\sigma_2\subset\sigma_1$, the metric $g_{\sigma_2}$ agrees with the
restriction of $g_{\sigma_1}$ to $\vec\sigma_2$. In such a case, we
 say that $g_\Sigma$
is a \emph{polyhedral metric} on $\Sigma$.

\begin{dfn}
A topological surface $\Sigma$ endowed with a triangulation
$\scrT=(\Sigma,\scrK,\ell)$ and a polyhedral metric $g_\Sigma$
is called a \emph{polyhedral surface}.
\end{dfn}

The polyhedral metric $g_\Sigma$ provides a flat Riemannian metric
 $g_\sigma$ on each simplex $\sigma\in\scrK$. As in the smooth case
 discussed at \S\ref{sec:fiberwise}, the metric $g_\sigma$
induces a fiberwise Euclidean inner product, denoted $g_\sigma$ as well,
on all the \emph{tensor bundles} over
the simplex
$\sigma$. Together with $g_V$, we deduce a Euclidean fiberwise
inner product for the bundles of $\vecV$-valued forms on $\sigma$, also
denoted $g_\sigma$.
For simplicity of notation, these inner products are sometimes denoted
$\ip{\cdot,\cdot}$ and the corresponding norm is denoted $|\cdot|$.

\subsection{Whitney cohomology}
 The theory of Whitney forms and Whitney cohomology~\cite{Whi} is a generalisation, in
 the piecewise linear setting, of  smooth differential forms and de Rham cohomology.
 The spaces of smooth differential forms $\Omega^n(\sigma)$ on a simplex
 $\sigma\in\scrK$ are compatible with pullbacks 
  in the following sense: for every
$\sigma_1,\sigma_2\in\scrK$ with
$\sigma_2\subset \sigma_1$ and $\beta_1\in\Omega^n(\sigma_1)$, the pullback
$\beta_2$ of $\beta_1$ on $\sigma_2$ is a
smooth differential form as well.

We denote by $\Omega^n(\Sigma,\scrT)$ the space of families of differential
form
$\beta=(\beta_\sigma)_{\sigma\in\scrK}$, where each
$\beta_\sigma\in\Omega^n(\Sigma)$.
Given  $\beta=(\beta_\sigma)\in\Omega^n(\Sigma,\scrT)$,
suppose that for every $\sigma_1,\sigma_2\in\scrK$, with
$\sigma_2\subset\sigma_1$, the
pull back of $\beta_{\sigma_1}$ agrees with $\beta_{\sigma_2}$. Then we say that
$\beta$ satifies the \emph{Whitney condition}, or that $\beta$ is  a \emph{Whitney
form}. 
The space of Whitney forms is  denoted $\Omega_w^n(\Sigma, \scrT)$.

There is a natural exterior derivative $d:\Omega^n(\Sigma, \scrT)\to
\Omega^{n+1}(\Sigma,\scrT)$ given by $d\beta =
(d\beta_\sigma)_{\sigma\in\scrK}$. Furthermore, the spaces of Whitney forms are
preserved by $d$ and we obtain the  Whitney complex
$$
d:\Omega^n_w(\Sigma,\scrT) \to\Omega_w^{n+1}(\Sigma,\scrT)
$$
which defines  the Whitney cohomology
denoted 
$H^n_w(\Sigma,\RR,\scrT)$. The constructions readily extend to the case of
$\vecV$-valued differential forms and we obtain the Whitney cohomology
spaces $H^n_w(\Sigma,\vecV,\scrT)$.
Similarly to the de Rham cohomology, the Whitney  cohomology
 agrees with the usual cohomology
spaces
$H^n(\Sigma,\RR)\simeq H^n_w(\Sigma,\RR,\scrT)$ and
$H^n(\Sigma,\vecV)\simeq H^n_w(\Sigma,\vecV,\scrT)$. In particular, the
Whitney cohomology
does not depend on the choice of triangulation $\scrT$.

\subsection{Orientation and Kähler structure}
An \emph{orientation} of the simplicial complex
$\scrK$ of a polyhedral surface $\Sigma$ is also called an orientation of $\Sigma$.
An \emph{oriented polyhedral surface carries} a canonical area form
$\omega_\Sigma=(\omega_\sigma)_{\sigma\in\scrK}\in\Omega^2_w(\Sigma,\scrT)$ 
characterized by 
\begin{enumerate}
	\item $\omega_\sigma$ is an exterior $2$-form on $\vec\sigma$.
	\item For every $\sigma\in\scrK_2$,  the volume form $\omega_\sigma$ is 
		compatible with the orientation of the facet and
		$|\omega_\sigma| =1$, with respect to $g_\sigma$.
		\item If $\sigma\in\scrK$ is an edge or a vertex then
			$\omega_\sigma=0$.
\end{enumerate}

The combination of $\omega_\sigma$ and $g_\sigma$ defines a family 
of almost complex structures 
$J_\Sigma=(J_\sigma)_{\sigma\in\scrK_2}$ on the facets of $\scrK$, where
$J_\sigma:\vec\sigma\to\vec\sigma$ is an almost complex structure on
$\vec\sigma$, with the property that
$\omega_\sigma = g_\sigma(J_\sigma\cdot,\cdot)$. 

In conclusion, an oriented polyhedral surface $\Sigma$ is endowed with
the structures
$(\scrT,g_\Sigma,J_\Sigma,\omega_\Sigma)$, referred to as a \emph{polyhedral
Kähler structure}, a notion introduced by
Dmitri Panov in~\cite{Pan}, where he investigates the $4$ dimensional case.
\begin{rmk}
A polyhedral Kähler surface $\Sigma$ as above admits a smooth structure
	\cite{Pan,Thu,Tro}
such that
\begin{enumerate}
	\item $g_\Sigma$ and $\omega_\Sigma$ are smooth away from the
		vertices of the triangulation.
	\item The metric may have  conical singularities at the vertices.
		\item The almost complex structure $J_\Sigma$ is smooth on
			$\Sigma$.
\end{enumerate}
In presence of conical singularities, the map $\ell:\sigma\to\Sigma$ is not smooth at the
vertices. In particular, the smooth differential forms on $\Sigma$ do not
	necessarily induce smooth families of Whitney forms in
	$\Omega^n_w(\Sigma,\scrT)$.
\end{rmk}

\subsection{Polyhedral maps and differentials}
We continue our constructions and analogies with the smooth setting, 
when $\Sigma$ is a closed oriented
polyhedral surface. 

A map $f:\Sigma\to V$, or more generally a map to some affine space,
is called \emph{polyhedral} with respect to $\scrT$ if:
\begin{enumerate}
	\item the map $f$ is continuous and
	\item the restriction of $f$ to every simplex of the triangulation
		$\scrT$ is an affine map.
\end{enumerate}
The space of polyhedral maps  is denoted 
$$
\scrM(\scrT)=\{f:\Sigma\to V, \mbox{ $f$ is $\scrT$-polyhedral} \}.
$$
The restriction of a polyhedral map $f$ to every simplex
$\sigma\in\scrT$ is affine, hence differentiable. However, $f$ is generally
not differentiable at every point of $\Sigma$. Nevertheless, we can define
the differential $d f$ as a family 
$$
d f =(d f|_\sigma)_{\sigma\in\scrK}.
$$
By construction, $df$ is a Whitney form
$df\in \Omega^1_w(\Sigma,\vecV,\scrT)$. The moduli space $\scrM(\scrT)$ is an
affine space with underlying vector space contained in
$\Omega^0_w(\Sigma,\scrT,\vecV)$ and it follows that $df$ is an
\emph{exact} Whitney form.

These observations motivate the introduction of a space of constant $\vecV$-valued
differentials, closely related to $\Omega^1(\Sigma,\vecV,\scrT)$:
$$
\scrF(\scrT)= \{(F_\sigma)_{\sigma\in\scrK_2}, 
F_\sigma\in \vec\sigma^*\otimes
\vecV\},
$$
where $\vec\sigma^*$ is the dual of the tangent direction $\vec\sigma$.
An element of $\vec\sigma^*\otimes \vecV$ is a linear map
$F_\sigma:\vec\sigma\to\vecV$ and 
this map can be regarded as a constant $\vecV$-valued differential $1$-form
on $\sigma$.

However, there are no compatibility conditions a priori along the edges of the
triangulation. We introduce  the subspace of \emph{Whitney
forms} which is a slightly different condition since $\sigma\in\scrK_2$: 
an element $F=(F_\sigma)_{\sigma\in\scrK_2}\in \scrF(\scrT)$ satisfies the \emph{Whitney
condition} if for every $\sigma_1,\sigma_2 \in\scrK_2$ with a common edge
$\sigma_3$, the pullbacks of $F_{\sigma_1}$ and $F_{\sigma_2}$ agree along
$\sigma_3$. In this case, $F$ is called a \emph{Whitney form} and we define:
$$
\scrF_w(\scrT)=\{F\in\scrF(\scrT),\quad F \mbox{ is a Whitney form}\}.
$$
Thanks to the Whitney condition, a Whitney form $F\in\scrF_w(\scrT)$ defines via the
pullback an extended
family $(F_\sigma)$ for $\sigma\in\scrK$ which is understood as an element
of $\Omega^1_w(\Sigma,\vecV,\scrT)$.
Hence, we have a
canonical embedding 
$$
\scrF_w(\scrT)\hookrightarrow\Omega^1_w(\Sigma,\vecV,\scrT).
$$
The elements  $F\in \scrF_w(\scrT)$ are families of constant differential
forms along the vertices of $\scrT$, hence they are \emph{closed}. Thus, every $F$
defines a Whitney cohomology class denoted $[F]\in H^1_w(\Sigma,\vecV)$.
We  define the subspace of exact forms in $\scrF_w(\scrT)$ by
$$
\scrF_0(\scrF)= \{F\in\scrF_w(\scrT),\quad [F]=0\}.
$$
By construction, the differential $d$ defines a map
$$
d:\scrM(\scrT)\to\scrF_0(\scrT).
$$
Furthermore the map $d$ induced an homeomorphism
$$d:\scrM(\scrT)/\vecV \to\scrF_0(\scrT).
$$
where $\vecV$ acts
by translations on $\scrM(\scrT)$.

\subsection{Moduli space structure in the polyhedral setting}
\label{sec:kahl}
The space $\scrF(\scrT)$ carries a Kähler structure defined similarly to 
the smooth setting. For $F=(F_\sigma)$ and $H=(H_\sigma)\in\scrF(\scrT)$,
we put
$$
\cG(F,H)=\sum_{\sigma\in\scrK_2}\int_\sigma
\ip{F_\sigma,H_\sigma}\omega_\sigma.$$
The integrands of the above integrals are constant on each simplex. Hence
each term of the above integral is equal to 
$\ip{F_\sigma,H_\sigma}\area(\sigma)$, where $\area(\sigma)$ is the area
of $\sigma$ with respect to the Euclidean metric $g_\sigma$.
As in the smooth case, we will use the notation $\ipp{F,H}=\cG(F,H)$ and
$\|F\|_{L^2}$ for the norm deduced from $\cG$.

An almost complex structure
$\cJ$ on $\scrF(\scrT)$ is given by
$$
(\cJ F)_\sigma= - F_\sigma \circ J_\sigma.
$$
By construction $\cJ$ is compatible with $\cG$ and we obtain a Kähler form
$$
\Omega = \cG(\cJ\cdot,\cdot).$$
Finally, we have defined a flat Kähler structure
$$
(\scrF(\scrT),\cG,\cJ,\Omega)
$$
on the moduli space.

As in the smooth case, the almost complex structure $i$ acts by
multiplication on $\vecV$-valued forms and we obtain a linear involution 
$\scrR$ on $\scrF(\scrT)$, defined by
$$
\scrR F = i\cJ F.
$$
The involution gives an orthogonal splitting into eigenspaces
$$
\scrF(\scrT)= \scrF^+(\scrT)\oplus \scrF^-(\scrT)
$$
and we introduce the $\cG$-orthogonal projection onto the subspace of exact
differentials
$$
\Pi:\scrF(\scrT)\to\scrF_0(\scrT).
$$

\subsection{Polyhedral symplectic density}
\label{sec:sympdens}
We denote denote 
by $\LieTT^\CC(\scrT)$ the space of maps
$\zeta:\scrK_2\to\CC$. Equivalently, an element $\zeta$ can be understood
as a family of constant maps $(\zeta_\sigma)_{\sigma\in\scrK_2}$ on each facet
$\sigma$ of the triangulation.
We also denote by $\LieTT(\scrT)$ the space of real valued maps
$\zeta:\scrK_2\to\RR$ and we define 
$$
\mu:\scrF(\scrT)\to \LieTT(\scrT)
$$
by the formula 
$$
\mu(F)(\sigma)=
-\frac 12 \ip{(\scrR F)_\sigma,F_\sigma}.
$$

We can make sense of the pullback $F^*\omega_V$ as a family of differential
$2$-forms along each facet  defined by
$$
(F^*_\sigma\omega_V)(\eta_1,\eta_2)= \omega_V(F_\sigma \cdot
\eta_1,F_\sigma \cdot \eta_2 )
$$
for every $\eta_1,\eta_2 \in \vec\sigma$, and $\sigma\in\scrK_2$.
As in the smooth case, the map $\mu$ is a symplectic density, in the sense
of the following lemma:
\begin{lemma}
	\label{lemma:denspol}
	For every $F\in\scrF(\scrT)$, we have the identity
	\begin{equation}
		\label{eq:denspol}
	\mu(F)(\sigma)=
	-\frac {F_\sigma^*\omega_V}  {\omega_\sigma}
	\end{equation}
for every facet $\sigma\in\scrK_2$.
\end{lemma}
This motivates the following definition:
\begin{dfn}
	A differential $F\in\scrF(\scrT)$ is called isotropic if
	$F_\sigma:\vec\sigma\to V$ is isotropic for every $\sigma\in\scrK_2$.
	A  polyhedral map $f\in\scrM(\scrT)$ is called isotropic if
	the pullback of $\omega_V$ by $f$ vanishes along every simplex of
	the triangulation.
\end{dfn}
In particular, by Lemma~\eqref{lemma:denspol}, we have
\begin{lemma}
	\label{lemma:isotropic}
	A differential $F\in\scrF(\scrT)$ is isotropic if, and only if,
	$\mu(F)=0$.
	A polyhedral map $f\in\scrM(\scrT)$ is isotropic if, and only if,
	$F=df$ is isotropic.
\end{lemma}

\subsection{Hamiltonian gauge group action}
We define the complex gauge group
$\TT^\CC(\scrT)$ as the space of non vanishing complex valued functions
$\lambda:\scrK_2\to\CC^*$.
Alternatively, $\lambda$ can be thought of as a family of constant
functions $(\lambda_\sigma)$ along each facet of the triangulation.
The group $\TT^\CC(\scrT)$ acts on $\scrF(\scrT)$ via the formula
\begin{equation}
	\label{eq:acpol}
	\lambda \cdot F = \bar\lambda^{-1}F^+ + \lambda F^-,
\end{equation}
where $F^\pm$ are the component of $F$ given by  the
splitting $\scrF(\scrT)=\scrF^+(\scrT)\oplus \scrF^-(\scrT)$.
The real subgroup $\TT(\scrT)$ is the the space of functions
$\lambda:\scrK_2\to S^1$ and acts by complex multiplication on
$\scrF(\scrT)$.
The exponential map 
 $$\exp:\LieTT^\CC(\scrT)\to
\TT^\CC(\scrT)$$
defined by
$$\exp \zeta = e^{i\zeta},$$
shows that $\LieTT^\CC(\scrT)$ is the Lie algebra of $\TT^\CC(\scrT)$ and 
identifies the  Lie algebra of the subgroup $\LieTT(\scrT)$ with $\LieTT(\scrT)$.
The infinitesimal action of $\LieTT^\CC(\scrT)$ on $\scrF(\scrT)$ is given by
$$
X_\zeta(F)= i\bar \zeta F^++i\zeta F^-.
$$
As in the smooth case, the construction has the following nice properties:
\begin{prop}
	The $\TT^\CC(\scrT)$-action on $\scrF(\scrT)$ perserves the almost complex
	structure $\cJ$. Furthermore, the $\TT^\CC(\scrT)$-action is the
	$\cJ$-complexification of the $\TT(\scrT)$-action.
The  $\TT$-action  preserves the metric $\cG$ and the
	symplectic form $\Omega$ as well.
	Furthermore the $\TT(\scrT)$-action is Hamiltonian with moment map 
	$\mu:\scrF(\scrT)\to\LieTT(\scrT)$ given by
	Formula~\eqref{eq:denspol}. More
	precisely, $\mu$ is $\TT(\scrT)$-invariant and for every
	$\zeta\in\LieTT(\scrT)$, we have
	$$
	D\ipp{\mu,\zeta} =
-	\iota_{X_\zeta}\Omega.
$$
\end{prop}
\begin{proof}
The proof is formally identical to the smooth setting.
\end{proof}

\subsection{Polyhedral modified moment map flow}
By analogy with the smooth setting, we consider the energy of the moment
map
$$\phi :\scrF(\scrT)\to\RR
$$
defined by
$$
\phi(F) = \frac 12 \|\mu(F)\|^2_{L^2}.
$$
By contruction, $\phi$ is non negative and its vanishing set is the
subspace of isotropic differentials in $\scrF(\scrT)$.
We denote by $\nabla \phi$ the gradient of $\phi:\scrF\to\RR$ with respect
to the metric $\cG$. The gradient of the restricted functional
$\phi:\scrF_0(\scrT)\to \RR$ is denoted $\nabla^\circ\phi$ and satisfies
$\nabla^\circ\phi(F)=\Pi(\nabla\phi(F))$.
Then we define the polyhedral modified moment map flow by the evolution
equation along $\scrF_0(\scrT)$:
\begin{equation}
	\label{eq:mmmff}
	\frac{\del F}{\del t}= - \nabla^\circ \phi(F).
\end{equation}
As in the smooth case, the formal identities are the same and we can prove:
\begin{prop}
	\label{prop:fixed}
	The critical points of the restricted functional
	$\phi:\scrF_0(\scrT)\to \RR$, in other words the fixed points of
	the modified moment map flow, agree with its vanishing locus.
	
	For every $F\in\scrF_0(\scrT)$, there exists a
	polyhedral map $f\in\scrM(\scrT)$ such that $F=df$, unique up to the action of
	$\vecV$ with the property that $f$ is isotropic if, and only if, $\phi(F)=0$. 
\end{prop}
\begin{rmk}
	\label{rmk:flowfunc}
The Second statement of the proposition implies that the modified moment
	map flow defines a flow for polyhedral maps as well: we fix a point
	$x_0\in\Sigma$ and a point $v_0\in V$. For every
	$F\in\scrF_0(\scrF)$, we define the integral
	$f=\chi(F)\in\scrM(\scrT)$ as the unique map such that $d f =
	F$ and $f(x_0)=v_0$. Any solution $F_t$ of the modified moment map
	flow defines a family $f_t=\chi(F_t)$ solution of the evolution
	equation
\begin{equation}
	\label{eq:liftflow}
	\boxed{
	\frac{\del f}{\del t} = -\chi\circ \nabla^\circ \phi ( d f) }
\end{equation}
and vice-versa.
	The fixed points of this flow are, by definition, the isotropic polyhedral maps.
\end{rmk}

	\begin{prop}
		\label{prop:decay}
		A solution of the polyhedral modified moment map flow
		$F_t\in\scrF_0(\scrT)$, defined for $t$ in some interval 
		satisfies
$$
		\frac \del{\del t} \|F\|^2_{L^2} = -8\phi(F_t).
$$
In particular, the $L^2$-norm is non increasing along the interval.
\end{prop}
\begin{proof}
The ingredients of the proof are the same as in Theorem~\ref{theo:flownorm},
	for the smooth setting. In the polyhedral context, we have the
	identity
\begin{equation}
	\label{eq:decpol}
	D\phi|_F\cdot F = 4\phi(F)
\end{equation}
	identical to the one in Proposition~\ref{prop:gradient} and the rest of the
	argument is identical.
\end{proof}

Proposition~\ref{prop:decay} has the following strong consequence:
\begin{cor}
	\label{cor:long}
	For every $F\in\scrF_0(\scrT)$, there exists a
	unique solution $F_t$ defined for $t\in [0,+\infty)$ and such that
	$F_0=F$.
	Furthermore, $\|F_t\|_{L^2}$ is bounded by $\|F\|_{L^2}$.
\end{cor}
\begin{proof}
	The short time existence is immediate Proposition \ref{prop:decay}
	and classical ODE theory. The decay of the
	$L^2$-norm implies that a solution of the flow cannot blow up in
	finite time, which implies the long time existence.
\end{proof}

\subsection{A Duistermaat theorem}
In fact, the polyhedral modified moment map flow converges and provides a
strong deformation retraction onto the
space of polyhedral isotropic maps.
\begin{theo}
	\label{theo:duis}
	For every $F\in\scrF_0(\scrT)$, the solution $F_t$ of the modified moment
	map flow, defined for $t\in[0,+\infty)$ with $F_0=F$ converges
	toward a limit $F_\infty$ such that $\phi(F_\infty)=0$.
Thus, we can define an extended flow
$$
	\Theta:[0,+\infty]\times\scrF_0(\scrT)\to\scrF_0(\scrT)
$$
	by $\Theta(t,F)=F_t$ and $\Theta(+\infty,F)=F_\infty$.
Furthermore, the map $\Theta$ is a strong deformation retraction onto the vanishing
	locus of $\phi:\scrF_0(\scrT)\to\RR$.
\end{theo}
\begin{proof}
	The proof is essentially contained in  the analogous  result
	\cite[Theorem 7.8.4]{Rol24}, where a downward gradient flow is also
	studied in a finite dimensional situation.
	Although the space $\scrF_0(\scrT)$, the functional $\phi:\scrF_0(\scrT)\to\RR$ and the
	Euclidean metric $\cG$ are completely different in
	\cite{Rol24}, they share similar formal properties:
	\begin{enumerate}
		\item The vector space is finite dimensional.\label{it:fd}
		\item The functional $\phi$ is a polynomial
			function.\label{it:pol}
		\item The functional $\phi$ is non negative.\label{it:pos}
		\item The critical points of $\phi$ agree with its
			vanishing set.\label{it:vs}
		\item The solutions of the downward gradient flow satisfy the
			decay property of Proposition~\ref{prop:decay}
			\label{it:decay}.
	\end{enumerate}
	Properties \eqref{it:fd} and \eqref{it:decay} imply the long time
	existence of the flow proved at Corollary~\ref{cor:long}.
	Property~\eqref{it:pol} implies that $\phi$ is analytic. In
	particular the \L{}ojaziewicz inequality can be applied to $\phi$,
	which is crucial to prove the convergence of the flow lines.
	The
	interested reader may check that only these properties are used in
	the proof of \cite[Theorem 7.8.4]{Rol24} and that they imply
	the congergence and continuity of the flow.
\end{proof}

\subsection{Regular points of the moduli space}
We obtain much stronger result for the behavior of the modified moment map flow in the
polyhedral setting, as proved by Theorem~\ref{theo:duis}. On the contrary, the
regularity of the space of isotropic maps is much easier to study in the
smooth setting, at least
in a neighborhood of an immersion, as showed at~\S\ref{sec:mb}. However,
some partial results can be obtained in the polyhedral case, as we are
going to see in this
section.
\begin{dfn}
	\label{dfn:regular}
	A  differential form $F\in\scrF_0(\scrT)$  is called regular, if
	the differential of 
	the restricted  moment map $\mu:\scrF_0(\scrT)\to\LieTT(\scrT)$ has constant
	rank in a neighborhood of $F$.
\end{dfn}
By definition, the vanishing locus of $\phi:\scrF_0(\scrT)\to\RR$ is a
submanifold of $\scrF_0(\scrT)$ in a neighborhood of a regular point $F$.
Furthermore, the following theorem shows that $\phi$ behaves like a
Morse-Bott function on such a neighborhood:
\begin{theo}
	\label{theo:reg}
	Let $F\in\scrF_0(\scrT)$ be a regular point of the vanishing locus
	of $\phi$. Then, there exists an open neighborhood
	$U\subset\scrF_0(\scrT)$ of $F$, such
	that: 
	\begin{enumerate}
	\item the vanishing locus of $\phi:U\to\RR$ is a submanifold
			of $U$;
		\item the Hessian of $\phi$ is positive definite in
			direction transverse to the vanishing locus;
		\item $U$ is invariant under the modified moment map flow
			$\Theta$, which has an exponential convergence rate.
	
	\end{enumerate}
\end{theo}
\begin{proof}
	The Hessian of $\phi$ is calculated 
	in the smooth
	setting at \S\ref{sec:mb} and is given a point $F$
	such $\phi(F)=0$ by Formula~\eqref{eq:hesssmooth}.
The same calculations show that the  formula holds in the polyhedral
	setting as well and obtain
	$$
D^2\phi|_F\cdot (\dot F_1,\dot F_2)= \ipp{D\mu|_F\cdot \dot
F_1,D\mu|_F\cdot \dot F_2} 
$$
	a any point $F$ of the vanishing locus of
	$\phi:\scrF_0(\scrT)\to\RR$.
If  $F$ is a regular, the vanishing locus of $\phi:\scrF_0(\scrT)\to\RR$ is
	a submanifold of $\scrF_0(\scrT)$ 
	in a neighborhood of $F$ and the above formula shows
	that the Hessian of $\phi$ is definite positive in transverse
	directions to the submanifold.
In conclusion, $\phi$ behaves like a Morse-Bott function in a neighborhood
of $F$ and its vanishing set is a stable critical component. The statements
of the theorem are classical facts of ODE theory under these assumptions.
\end{proof}

Constructing regular points of the moduli
space is not a simple problem. We are going to show that
Theorem~\ref{theo:approx} provides a construction of regular points. 
In \cite{JRT}, a quotient $2$-torus
$\Sigma=\RR^2/\Gamma$ is considered together with  
a smooth isotropic immersion $f:\Sigma\looparrowright\RR$. 
Refinements of stepsize $\scrO(N^{-1})$ of 
the lattice $\Gamma$ provide a family of quadrangulations $\scrQ_N$ of $\Sigma$, with $N^2$ facets.
The quadrangulations can be completed into triangulations $\scrT_N$ by replacing each
quadrilateral of $\scrQ_N$ with a pyramid.
The main construction of \cite{JRT} provides a sequence of  polyhedral isotropic maps
$f_N\in\scrM(\scrT_N)$, for every $N$ sufficiently large, such that $f_N$
approximates $f$ in $C^1$-norm  by~\cite{Rol22}.
The construction of $f_N$ relies on an effective version of the fixed point principle
applied to the family of symplectic densities
$$
\mu_N\circ d:\scrM(\scrT_N)\to\LieTT(\scrT_N),
$$
where $\mu_N:\scrF_0(\scrT_N)\to \LieTT(\scrT_N)$ is the moment map. 
The fact that the cohomology class of $\omega_V$ vanishes implies that
$\mu_N\circ d $ takes values in the subspace of $\zeta\in\LieTT(\scrT_N)$
orthogonal to constant functions. Hence $\mu_N\circ d$ can never be a
submersion, but it may have maximal rank with corank one.
A careful examination of
the proof in \cite{JRT} allows to make the following observation:
\begin{lemma}
	\label{lemma:reg}
For every $N$ sufficiently large, the differential of the map $
	\mu_N\circ d:\scrM(\scrT_N)\to \LieTT(\scrT_N)$ has maximal rank
	at $f_N$, where $f_N$ is the Jauberteau-Rollin-Tapie approximation
	of $f$. In particular, $F_N=df_N\in\scrF_0(\scrT_N)$ is a regular
	point of the vanishing set of $\phi:\scrF_0(\scrT_N)\to\RR$.
\end{lemma}
We can now complete the proofs of the main theorems given in the
introduction.
\begin{proof}[Proof of Theorem~\ref{theo:mainbis}]
	The result is an immediate consequence of Lemma \ref{lemma:reg}.
\end{proof}

	\begin{proof}[Proof of Theorem~\ref{theo:main}]
	The result is a combination of Proposition~\ref{prop:fixed}, Theorem~\ref{theo:duis} and
		Theorem~\ref{theo:reg}. The statement $(3)$ will be proved
		at  the end of \S\ref{sec:topo}.
\end{proof}
\begin{proof}[Proof of Corollary~\ref{cor:main}]
	The differential $d:\scrM(\scrT)\to\scrF_0(\scrT)$  admits a unique right inverse
$$
\chi:\scrF_0(\scrT)\to\scrM(\scrT)
$$
defined by the condition that $d\circ\chi= \id$ and 
	$f=\chi(F)$ satisfies $f(x_0)=v_0$ for some
marked points $x_0\in\Sigma$ and $v_0\in V$.
We consider the continuous map
$$
\hat \Theta(t,f)= \chi\circ \Theta(t,df)+ (f(x_0)-v_0).
$$
By construction
$$
\hat\Theta:[0,+\infty]\times\scrM(\scrT)\to\scrM(\scrT)
$$
defines a strong deformation retraction of $\scrM(\scrT)$ onto the subspace of isotropic
polyhedral maps. It is easy to see that $\hat\Theta$ is  the flow
of the evolution equation on $\scrM(\scrT)$ defined at
Remark~\ref{rmk:flowfunc}. In particular $\Theta_t\circ d =
	d\circ\hat\Theta_t$, which proves the corollary.
\end{proof}

\subsection{Topology of the moduli space}
\label{sec:topo}
The homotopy  of the space of smooth isotropic immersions
$f:\Sigma\looparrowright V$ can be reduced to a matter of algebraic
topology,
thanks to the Gromov-Lees theorem~\cite{Gro86,Lee76}: every isotropic immersion $f$ defines
an exact differential form $F=df\in\Omega^1(\Sigma,\vecV)$ which is by
definition an
 isotropic monomorphism.
The Gromov-Lees theorem establishes an  h-principle, which shows that the
differential $d:\scrM/\vecV\to\scrF_0$
induces a homotopy equivalence from the space of isotropic immersions to
the space of isotropic monomorphisms. Computing the homotopy of the latter
space turns out to be a basic problem of algebraic topology, related to
the Grassmanian
of isotropic planes in $\vecV$.

Unfortunately, the Gromov-Lees theory does not extend immediately to
the polyhedral setting. Thus, most topological properties for 
the space of polyhedral isotropic
immersions are open questions although the partial results of
\cite{JRT,Rol22,Eto} lead us to state
Conjecture~\ref{conj:intro}.

The map $d:\scrM(\scrT)/\vecV \to\scrF_{0}(\scrT)$ is a homeomorphism, that
identifies the subspace of polyhedral isotropic maps modulo $\vecV$ with
the space of exact isotropic differentials. 
 The latter space is a cone in
$\scrF_0(\scrT)$, which implies that it is contractible. Therefore,
the space of polyhedral isotropic maps is homotopically trivial. The same
property holds in the smooth setting, which is why the problem of
immersions is
considered instead. There are several issues here to formulate analogous
questions in the polyhedral setting:
\begin{enumerate}
	\item A polyhedral map
$f\in\scrM(\scrT)$ which is a topological immersion define a monomorphism
$F=df\in\scrF_0(\scrT)$. However the converse is generally not true and
some local injectivity condition along the skeleton of the triangulation
should be added to obtain a correspondence.
\item 
The space of monomorphisms in $\scrF_0(\scrT)$ does not seem to be
preserved by the modified moment map flow $\Theta$. From this point, the prospect of 
using the flow to
investigate the homotopy type of the space of polyhedral isotropic
immersions seems rather low.
\end{enumerate}
However the space of non constant polyhedral isotropic maps may already contain
some non trivial topology. Furthemore, this subspace modulo $\vecV$ is
identified to $\scrF_0(\scrT)\setminus 0$ via $d$. Finally, $\scrF_0(\scrT)\setminus 0$
is invariant under the gauge group action and invariant under the modified
moment map flow, for finite time.
\begin{question}
	\label{q:homot}
	What is the homotopy type of the space of non constant polyhedral isotropic maps
	$f\in\scrM(\scrT)$~?
\end{question}
The  space of isotropic exact differentials in $\scrF_0(\scrT)\setminus 0$ is
invariant by scaling. In particular, it is homotopically
equivalent to the vanishing set of the restriction $\phi:\SS(\scrT)\to\RR$, where
$\SS(\scrT)$ is the unit sphere in $\scrF_0(\scrT)$.

The idea to tackle Question~\ref{q:homot}  is to interpret $\phi:\SS(\scrT)\to\RR$ as a Morse-Bott function
to obtain some dynamical interaction between its critical components. By
definition the critical set of $\phi:\SS(\scrT)\to\RR$ contains its
vanishing set, but there may be other critical points,
called the \emph{solitons}. More explicitly, a soliton $F\in\scrF_0(\scrT)$ is
a solution of  the  equation
$
\nabla^\circ\phi(F)=\kappa F$
for some $\kappa\in\RR$. The constant $\kappa$  is determined by
Formula~\ref{eq:decpol} and we obtain the soliton equation:
\begin{equation}
	\label{eq:soliton}
	\boxed{
\|F\|^2_{L^2}\nabla^\circ\phi(F) = 4\phi(F). 
}
\end{equation}
Solitons are by definition the fixed points of  the downward gradient flow of the
restricted functional $\phi:\SS(\scrT)\to\RR$ defined for $F\in\SS(\scrT)$ by
\begin{equation}
	\label{eq:ren}
	\boxed{
		\frac{\del F}{\del t}= 4\phi(F) -\nabla^\circ\phi(F)  
}
\end{equation}
and called the \emph{renormalized} polyhedral flow.

Many questions are open at this stage for the renormalized flow and the
solitons:
\begin{enumerate}
	\item Is $\phi:\SS(\scrT)\to \RR$ a Morse-Bott type function, in
		some sense, perhaps up to
		the choice a refinement of the  triangulation and a generic polyhedral metric~?
	\item What are the intrinsic geometrical properties of polyhedral maps
		$f\in\scrM(\scrT)$ such that $F=df$ is a  soliton ?
	\item Does the functional $\phi:\SS(\scrT)\to\RR$ define a
		Morse-Bott cohomology theory ?
\end{enumerate}

We can prove that solitons never reduce to the vanishing locus of $\phi$.
\begin{lemma}
	There exists solitons $G\in\scrF_0(\scrT)$ such that $\phi(G)\neq
	0$.
\end{lemma}
\begin{proof}
Clearly, $\phi:\SS(\scrT)\to\RR$ is not the zero map, as it is easy to
construct polyhedral maps $f\in\scrM(\scrT)$ which are not isotropic. In
particular $\phi(df) > 0$ in this case. We conclude that $\phi$ admits a
maximum at some point $G\in\SS(\scrT)$ with $\phi(G)>0$. If follows that
$G$ is a critical point of $\phi:\SS(\scrT)\to\RR$, which is to say a soliton. Furthermore,
$\phi(G) >0$ implies that $G$ is not isotropic.
\end{proof}

\begin{rmk}
	We introduce a family $F_t\in\scrF_0(\scrT)$ given by $F_t=r(t)G$, where
	$G\in\scrF_0(\scrT)$ is a soliton such that $\phi(G)>0$ and
	$$
r(t) = \frac 1{\sqrt{8(t-t_0)\phi(G)}}
$$
is defined on the interval $(t_0,+\infty)$, for some choice of $t_0\in\RR$.
It is readily checked that $F_t$ is a solution of the polyhedral modified
moment map flow. We notice that $F_t$ defines a non trivial solution of
flow, with the property that $\displaystyle \lim_{t\to+\infty}F_t=0$. 
\end{rmk}

\begin{proof}[Proof of Theorem~\ref{theo:main}, item (3)]
	The existence of non isotropic solitons in $\SS(\scrT)$ provides
	non trivial solutions of the polyhedral modified moment map flow as
	discussed above.
	This completes the proof of the theorem.
\end{proof}

\subsection{Numerical flow}
\label{sec:numflow}
The polyhedral modified moment map flow 
is an ordinary differential evolution equation. A numerical
version of the flow can be implemented on a computer. In this section
we outline all the  ingredients for such a computer program.
The code is being developed currently and is going to be released
very soon~\cite{JR1}.

For simplicity, we specialize to the case where $\Sigma$ is
 a quotient $2$-torus.
Let $\Gamma$ be the lattice  
$$
\Gamma = \gamma_1 \ZZ\oplus \gamma_2\ZZ \subset \CC
$$
where  $\gamma_1=1$ 
$\gamma_2=e^{i\frac\pi 3}$. The surface $\Sigma$,  defined by
$$
\Sigma=\CC/\Gamma,
$$
is endowed with the flat Kähler structure
$(\Sigma,g_\Sigma,J_\Sigma,\omega_\Sigma)$ deduced from the
canonical flat Kähler structure of $\CC$. 
The rhombus in $\CC$ with
vertices $0, \gamma_1,\gamma_1 +\gamma_2$ and $\gamma_2$ is a fundamental
domain for the action of $\Gamma$.
This rhombus  is composed of two equilateral triangles and each triangle has
area $\frac{\sqrt 3}4$. Using the $\Gamma$-action, 
we obtain the familiar tiling of $\CC$ by equilateral triangles:
\begin{figure}[H]
	\begin{tikzpicture}[scale=.5]
		\draw (0,0) -- (5,8.66) -- (15,8.66) -- (10,0) -- (0,0);
		\draw (1, 0) -- (6, 8.66);
		\draw (2, 0) -- (7, 8.66);
		\draw (3, 0) -- (8, 8.66);

		\draw (4.0, 0) -- (9.0, 8.66);
		\draw (5.0, 0) -- (10, 8.66);

		\draw (6.0, 0) -- (11, 8.66);
		\draw (7.0, 0) -- (12, 8.66);
		\draw (8.0, 0) -- (13, 8.66);
		\draw (9.0, 0) -- (14, 8.66);
		\draw (1, 0) -- (.5, 0.866);

		\draw (2, 0) -- (1, 1.732);
		\draw (3, 0) -- (1.5, 2.598);
		\draw (4, 0) -- (2, 3.464);
		\draw (5, 0) -- (2.5, 4.330);
		\draw (6, 0) -- (3, 5.196);
		\draw (7, 0) -- (3.5, 6.062);
		\draw (8, 0) -- (4, 6.928);
		\draw (9, 0) -- (4.5, 7.794);
		\draw (10, 0) -- (5, 8.66);
		\draw (10.5, 0.866) -- (6, 8.66);

		\draw (11.0, 1.732) -- (7, 8.66);
		\draw (11.5, 2.598) -- (8, 8.66);
		\draw (12.0, 3.464) -- (9, 8.66);
		\draw (12.5, 4.330) -- (10, 8.66);
		\draw (13.0, 5.196) -- (11, 8.66);
		\draw (13.5, 6.062) -- (12, 8.66);

		\draw (14.0, 6.928) -- (13, 8.66);
		\draw (14.5, 7.794) -- (14, 8.66);
		\draw (10.5, 0.866) -- (.5, 0.866);
		\draw (11, 1.732) -- (1, 1.732);
		\draw (11.5, 2.598) -- (1.5, 2.598);
		\draw (12, 3.464) -- (2, 3.464);
		\draw (12.5, 4.330) -- (2.5, 4.330);
		\draw (13, 5.196) -- (3, 5.196);
		\draw (13.5, 6.062) -- (3.5, 6.062);
		\draw (14, 6.928) -- (4, 6.928);
		\draw (14.5, 7.794) -- (4.5, 7.794);
	\end{tikzpicture}
  \caption{Triangulation of $\CC$ by equilateral triangles}
	\label{fig:tri}
\end{figure}
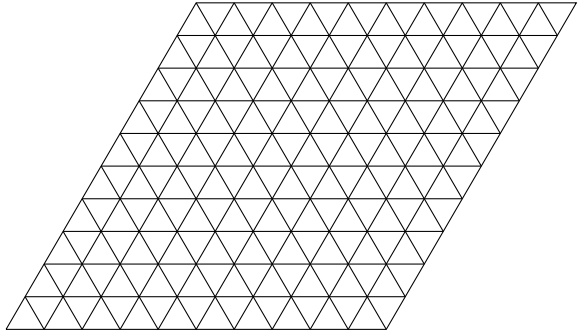
This tiling is a  simplicial decomposition of $\CC$ and can be
regarded as a triangulation $\scrT'$ of $\CC$.
Let $N$ be a positive integer. We define the lattice
$$
\Gamma_N = N^{-1} \Gamma,
$$
understood as the set of vertices of the rescaled  triangulation $\scrT'_N$
of $\scrT'$, by a factor $N^{-1}$.
 By
Definition, the triangulation $\scrT'_N$  is $\Gamma$-invariant. Hence, $\scrT'_N$ defines a triangulation of the quotient
$\Sigma$, denoted
$\scrT_N$, with each triangle of area $\frac{\sqrt 3}{4N^2}$.

We consider the particular space $V=\CC^2$ 
endowed with its canonical structure of Hermitian affine space and  the
coordinates $(x_1+ix_2,x_3+ix_4)\in V$, where $x_i\in\RR$.
 The symplectic form $\omega_V$ is given in coordinates by the formula
$$
\omega_V=dx_1\wedge dx_2+dx_3\wedge dx_4
$$
and the almost complex structure $i:\vecV\to\vecV$, given  by the
multiplication by $i$ satisfies the identities
\begin{equation}
	\label{eq:cxx}
i \frac\del{\del x_1}= \frac\del{\del x_2} \quad \mbox { and }\quad   i\frac\del{\del x_3}
=\frac\del{\del x_4}.
\end{equation}

Every $F\in\scrF(\scrT_N)$ is given by a collection
$F=(F_\sigma)_{\sigma\in\scrK_2}$ where $F_\sigma\in\vec \sigma^*\otimes\vecV$.
We use the canoncial coordinates $z=u_1+i u_2\in \CC$ on the universal
cover $\CC$ 
of $\Sigma$. 
By definition we have
\begin{equation}
	\label{eq:cxu}
-du_1\circ J_\Sigma = du_2 \mbox{ and } -du_2\circ J_\Sigma = -du_1
\end{equation}
The $du_i\otimes\frac\del{\del x_j}$ provide a $g_\sigma$-orthonormal
basis of $\vec\sigma^*\otimes\vecV$ for each facet of the
triangulation $\scrT_N$. 
Hence, $F_\sigma$ admits a decomposition  
\begin{equation}
	\label{eq:coord}
	F_\sigma  = 
	\sum_{\substack{1\leq i \leq 2\\
		1\leq j\leq 4}} 
	F_{\sigma ij}  du_i\otimes
\frac\del{\del x_j}
\end{equation}
where $F_{\sigma ij}\in\RR$. 
In particular, for $F, H\in\scrF(\scrT_N)$, we have
$$
\ipp{F,H} = \sum_{\substack{1\leq i \leq 2\\
1\leq j\leq 4\\
\sigma }} F_{\sigma ij} H_{\sigma ij} \frac{\sqrt
3}{4N^2}
$$
where the sum is taken over all the facets $\sigma$ of $\scrT_N$.

The space $\scrF_0(\scrT_N)$ is the subspace of exact
$\vecV$-valued Whitney forms in $\scrF(\scrT_N)$.
We need to construct an orthonormal basis of this subspace to compute the
matrix of the projector $\Pi:\scrF(\scrT_N)\to\scrF_0(\scrT_N)$.

We consider the  affine canonical base of $V$ given by the points
$e_0=0$ and $e_j= e_0+\frac\del{\del x_j}$ for $1\leq j\leq 4$.
For each vertex $\sigma_0$ of $\scrT_N$ and $1\leq j\leq 4$, we define the polyhedral
 function $f^{\sigma_0 j}:\Sigma\to V$
given by
$$
f^{{\sigma_0}j}(\sigma_0')=
\left \{
	\begin{array}{ll}
		e_j& \mbox{if $\sigma_0=\sigma'_0$,}\\
		e_0&\mbox{otherwise.}
	\end{array}
\right .$$
for every vertex $\sigma_0'$ of $\scrT_N$.
The above function extends uniquely as a polyhedral function
$f^{\sigma_0 j}\in\scrM(\scrT_N)$
and  defines the exact differential 
$$
F^{\sigma_0 j} = df^{ \sigma_0 j} \in\scrF_0(\scrT_N).
$$
The family $F^{\sigma_0 j}$ for $\sigma_0\in\scrK_0$ and $1\leq j\leq 4$ spans $\scrF_0(\scrT_N)$, by definition.
The only relations between the $F^{\sigma j}$ are the $4$ relations given
by
$$
\sum_{\sigma_0} F^{\sigma_0 j}=0
$$
for $1\leq j\leq 4$,
where the sum is taken over all the vertices of the triangulation
$\scrT_N$.
We obtain the following lemma:
\begin{lemma}
	\label{lemma:basis}
	For a fixed vertex $\sigma'_0$ of the triangulation $\scrT_N$, The family 
	$$
	F^{ \sigma_0 j}\in\scrF_0(\scrT_N),$$
	where $1\leq j\leq 4$
	and $\sigma_0$ belongs to the vertices of $\scrT_N$ with $\sigma_0\neq
	\sigma'_0$,
	is a basis of $\scrF_0(\scrT_N)$.
\end{lemma}
However, the family $F^{\sigma_0 j}$ is not orthonormal and we have to use
the Gram-Schmidt algorithm.
The
differential $F^{\sigma_0 j}$ can be computed explicitly,
working on the universal cover $\CC$ of $\Sigma$, and assuming $\sigma_0=0$
for simplicity. There are six
triangles of $\scrT'_N$ with vertex $\sigma_0=0$. We can compute the differential of
$f^{\sigma_0 j}$ on each triangle.  In particular on 
the triangle $\sigma_1$  with vertices $0$, $\gamma_1/N$ and
$\gamma_2/N$ we have
\begin{equation}
	\label{eq:inp1}
F^{\sigma_0 j}_{\sigma_1} =\left (-N d u_1 - \frac {N}{\sqrt 6} d u_2 \right) 
\otimes\frac{\del }{\del x_j}.
\end{equation}
The differential on all the other triangles with vertex $\sigma_0=0$ is obtained by rotations of angle
$\frac\pi 3$. For every other triangle $\sigma$ not in the star of
$\sigma_0$, we have $F_\sigma^{\sigma_0 j}=0$.
In particular, on the triangle $\sigma_2$ with vertices $0$, $-\gamma_1/N$ and
$e^{2i\frac \pi 3}/N$, we
have the formula
\begin{equation}
	\label{eq:inp2}
F^{\sigma_0 j}_{\sigma_2} =\left (N d u_1 - \frac {N}{\sqrt 6} d u_2 \right) 
\otimes\frac{\del }{\del x_j}.
\end{equation}
\begin{lemma}
	\label{lemma:ipF}
	The family of Whitney forms 
	$F^{\sigma_0 j}\in\scrF_0(\scrT_N)$ described above
	satisfies the identities
	\begin{align*}
		\ipp{F^{\sigma_0 j},&F^{\sigma_0' j'}}=\\
		&\left\{
		\begin{array}{cl}
			\frac {7\sqrt 3}4 &\mbox{ if $i=i'$ and
			$\sigma_0=\sigma_0'$,}\\
			-\frac{5}{4\sqrt 3} & \mbox{if $i=i'$ and
			$\sigma_0$ and $\sigma_0'$ are joined by an
			edge,}\\
			0 & \mbox{otherwise.}
		\end{array}
		\right .
	\end{align*}
\end{lemma}
\begin{proof}
	The third case is clear, since the families
	$(F_\sigma^{\sigma_0 j})$
	and $(F_\sigma^{\sigma_0' j'})$ have disjoint supports for $j\neq
	j'$ or if $\sigma_0$ and $\sigma'_0$ do not belong to the same
	facet of $\scrT_N$.

	We assume now that, on the contrary,
	$j=j'$ and $\sigma_0$ and $\sigma'_0$ belong to the same facet.
	Then either $\sigma_0=\sigma'_0$ or $\sigma_0$ and $\sigma'_0$ are
	the two ends of an edge  in $\scrT_N$.

	In the first case, working on the universal cover and assuming
	$\sigma_0=0$, Formula~\eqref{eq:inp1} 
	shows that $|F^{\sigma_0 j}_{\sigma_1}|^2= \frac{7N^2}6$. Using the
	fact that the area of a facet is $\frac{\sqrt 3}{4N^2}$ and that
	there are $6$ facets in the star of $\sigma_0$, we find
\begin{equation*}
\|F^{\sigma_0 j}\|^2_{L^2}=\frac{7\sqrt 3}4.
\end{equation*}
In the second case, working on the universal cover, we may assume that
	$\sigma_0'=\frac{\gamma_1}N$. By Formulas \eqref{eq:inp1} 
	and~\eqref{eq:inp2}, we deduce that 
	$\ip{F^{\sigma_0 j}_{\sigma_1},
	F^{\sigma_0' j}_{\sigma_1}}=-\frac{5N^2}6$.
	There is a second facet of $\scrT_N$ in the common support of $F^{\sigma_0 j}$
	and $F^{\sigma_0' j}$. By symmetry, the inner product is the same
	on the second facet. Since the area of each triangle is equal to
	$\frac{\sqrt 3}{4N^2}$, we obtain
\begin{equation*}
\ipp{F^{\sigma_0 j},F^{\sigma_0' j}}
	=-\frac{5N^2}{6}\cdot 2\cdot \frac{\sqrt 3}{4N^2}
	= -\frac 5{4\sqrt 3}.
\end{equation*}
\end{proof}
The family $F^{\sigma_0 j}$
provides a basis of $\scrF_0(\scrT_N)$ thanks to
Lemma~\ref{lemma:basis} and
the Gram-Schmidt method gives an algorithm  to construct an orthonomal
basis on $\scrF_0(\scrT_N)$ from this data.
Furthermore, the Gram-Schmidt method is pretty
straightforward regarding computer implementation thanks to the explicit
formulas of
Lemma~\ref{lemma:ipF}.

Eventually,
the matrix of the  endomorphism 
$$\scrR:\scrF_N(\scrF_N)\to\scrF(\scrT_N)$$
can be explicitely computed in the coordinates 
given by \eqref{eq:coord}, using Formula~\eqref{eq:cxu} and Formula~\eqref{eq:cxx} as
$\scrR F = -iF\circ J_\Sigma$, by definition.
Thus, we have an algorithm to compute
$\mu(F)= -\frac 12\ip{\scrR F,F}$
as well as
$\nabla \phi(F)= -\mu(F)\scrR F$.
We also have a matrix representation for $\Pi$ via Gram-Schmidt and we can compute
$\nabla^\circ \phi(F) =\Pi(\nabla(F))$.
It is now easy to find approximate solutions of the polyhedral modified
moment map flow~\eqref{eq:mmmff} via the Euler method. Similarly we can apply the Euler
approximation method to the polyhedral renormalized flow~\eqref{eq:ren}.

\vspace{10pt}
\bibliographystyle{abbrv}
\bibliography{dislagflow}

\end{document}